\documentclass{paper}

\usepackage{tikz}
\usepackage[T1]{fontenc}
\usepackage[english]{babel}

\usepackage{units}
\usepackage[authoryear]{natbib}%
\usepackage{amsthm,amsmath,amsbsy,amstext,amssymb}
\usepackage{color}

\usepackage{setspace}
  \theoremstyle{remark}
  \newtheorem{rem}{Remark}
\theoremstyle{plain}
\newtheorem{thm}{Theorem}
  \theoremstyle{plain}
  \newtheorem{cor}{Corollary}
  \theoremstyle{plain}
  \newtheorem{prop}{Proposition}
  \theoremstyle{plain}
  \newtheorem{lem}{Lemma}

\newcommand\E{\mathbb{E}}
\newcommand\p{\mathbb{P}}
\newcommand\ind{\mathbb{I}}

\newcommand\Xc{\mathcal{X}}
\newcommand\Ac{\mathcal{A}}

\global\long\def\E{\mathbb{E}}
\global\long\def\p{\mathbb{P}}

\global\long\def\Sc{\mathcal{S}}
\global\long\def\Xc{\mathcal{X}}
\global\long\def\Bc{\mathcal{B}}
\global\long\def\Ac{\mathcal{A}}
\global\long\def\ind{\mathbb{I}}
\global\long\def\go{\mathcal{O}}

\global\long\def\R{\mathbb{R}}

\title{Bayesian optimal adaptive estimation using a sieve prior\\
}
{\small
\author{Julyan Arbel\textsuperscript{1,2}, Ghislaine Gayraud\textsuperscript{1,3} and Judith Rousseau\textsuperscript{1,4}\\
\textsuperscript{1}\small{Laboratoire de Statistique, CREST, France}\\
\textsuperscript{2}\small{INSEE, France} \\
\textsuperscript{3}\small{LMAC, Universit\'e de Technologie de Compi\`egne, France}\\
\textsuperscript{4}\small{Universit\'e Paris Dauphine, France}}}

\begin{document}

\maketitle
\begin{abstract}
We derive rates of contraction of posterior distributions on
nonparametric models resulting from sieve priors. The aim of the
paper is to provide general conditions to get posterior rates when
the parameter space has a general structure, and rate adaptation
when the parameter space is, \textit{e.g.}, a Sobolev class. The
conditions employed, although standard in the literature, are
combined in a different way. The results are applied to density,
regression, nonlinear autoregression and Gaussian white noise
models. In the latter we have also considered a loss function which
is different from the usual $l^2$ norm, namely the pointwise loss.
In this case it is possible to prove that the adaptive Bayesian
approach for the $l^2$ loss is strongly suboptimal and we provide a
lower bound on the rate.
\end{abstract}


\section{Introduction}

The asymptotic behaviour of  posterior distributions in nonparametric
models has received growing consideration in the literature over the
last ten years. Many different models have been considered, ranging
from the problem of density estimation in i.i.d. models
\citep{Barron:1999p213,Ghosal:2000p242}, to sophisticated dependent
models \citep{ju_2010}. For these models, different families of
priors have also been considered, where the most common are
Dirichlet process mixtures (or related priors), Gaussian processes
\citep{van_der_vaart_rates_2008}, or series expansions on a basis
(such as wavelets, see \citealp{abramovich1998wavelet}).

In this paper we focus on a family of priors called \emph{sieve
priors}, introduced as \emph{compound priors} and discussed by
\citet{zhao1993frequentist,Zhao:2000p98}, and further studied by
\citet{Shen:2001p194}. It is defined for models
$(\Xc^{(n)},A^{(n)},P_{\boldsymbol{\theta}}^{(n)}:\boldsymbol{\theta}\in\Theta)$,
$n\in\mathbb{N}\backslash \{0\}$, where $\Theta
\subseteq\R^\mathbb{N}$, the set of sequences. Let $A$ be a
$\sigma$-field associated to $\Theta$. The observations are denoted
$X^{n}$, where the asymptotics are driven by $n$. The probability
measures $P_{\boldsymbol{\theta}}^{(n)}$ are dominated by some
reference measure $\mu$, with density
$p_{\boldsymbol{\theta}}^{(n)}$. Remark that such an
infinite-dimensional parameter $\boldsymbol{\theta}$ can often
characterize a functional parameter, or a curve,
$\boldsymbol{f}=\boldsymbol{f}_{\boldsymbol{\theta}}$. For instance,
in regression, density or spectral density models, $\boldsymbol{f}$
represents a regression function, a log density or a log spectral
density respectively, and $\boldsymbol{\theta}$ represents its
coordinates in an appropriate basis
$\boldsymbol{\psi}=(\psi_{j})_{j\geq 1}$ (\textit{e.g.}, a Fourier,
a wavelet, a log spline, or an orthonormal basis in general). In
this paper we study frequentist properties of the posterior
distributions as $n$ tends to infinity, assuming that data $X^{n}$
are generated by a measure $P_{\boldsymbol{\theta}_0}^{(n)}$,
$\boldsymbol{\theta}_{0}\in\Theta$. We study in particular rates of
contraction of the posterior distribution and rates of convergence
of the risk.

A sieve prior $\Pi$ is expressed as
\begin{equation}
\boldsymbol{\theta}\sim\Pi(\,\cdot\,)=\sum_{k=1}^\infty\pi(k)\Pi_{k}(\,\cdot\,),\label{eq:sievep}\end{equation}
where $\sum_{k}\pi(k)=1$, $\pi(k)\geq0$, and the $\Pi_{k}$'s are
prior distributions on so-called sieve spaces $\Theta_k=\R^{k}$. Set
$\boldsymbol{\theta}_k=(\theta_1,\ldots,\theta_k)$ the
finite-dimensional vector of the first $k$ entries of
$\boldsymbol{\theta}$. Essentially, the whole prior $\Pi$ is
seen as a hierarchical prior, see Figure \ref{fig:sieve}. The hierarchical parameter $k$, called
threshold parameter, has prior $\pi$. Conditionally on $k$, the
prior on $\boldsymbol{\theta}$ is $\Pi_{k}$ which is supposed to
have mass only on $\Theta_k$ (this amounts to say that the priors on
the remaining entries $\theta_{j}$, $j>k$, are point masses at 0).
We assume that $\Pi_{k}$ is an independent prior on the coordinates
$\theta_j$, $j=1,\ldots,k$, of $\boldsymbol{\theta}_{k}$ with a
unique probability density~$g$ once rescaled by positive
 $\boldsymbol{\tau}=(\tau_{j})_{j\geq 1}$. Using the same notation $\Pi_{k}$ for probability and density with Lebesgue measure or $\R^k$,
we have
\begin{equation}
\forall\boldsymbol{\theta}_k\in\Theta_k,\quad\Pi_{k}\left(\boldsymbol{\theta}_{k}\right)=\prod_{j=1}^{k}\frac{1}{\tau_{j}}g\left(\frac{\theta_{j}}{\tau_{j}}\right).\label{eq:gn}\end{equation}
Note that the quantities $\Pi,\,\Pi_{k},\,\pi,\,\boldsymbol{\tau}$
and $g$ could depend on $n$. {Although not purely Bayesian, data dependent priors are quite common in the literature. For instance, \cite{Ghosal:2007p267} use a similar prior with a deterministic cutoff $k=\lfloor n^{1/(2\alpha+1)} \rfloor$ in application 7.6.}\medskip

{We will also consider the case where the prior is truncated to an $l^1$ ball of radius $r_1>0$ (see the nonlinear AR(1) model application in Section \ref{AR1}) 
\begin{equation}
\forall\boldsymbol{\theta}_k\in\Theta_k,\quad\Pi_{k}\left(\boldsymbol{\theta}_{k}\right)\propto\prod_{j=1}^{k}\frac{1}{\tau_{j}}g\left(\frac{\theta_{j}}{\tau_{j}}\right)\mathbb{I}(\sum_{j=1}^k \vert \theta_j \vert \leq r_1).
\label{eq:gntrunc}
\end{equation}
}
\begin{figure}
\begin{center}
\includegraphics[width=3cm]{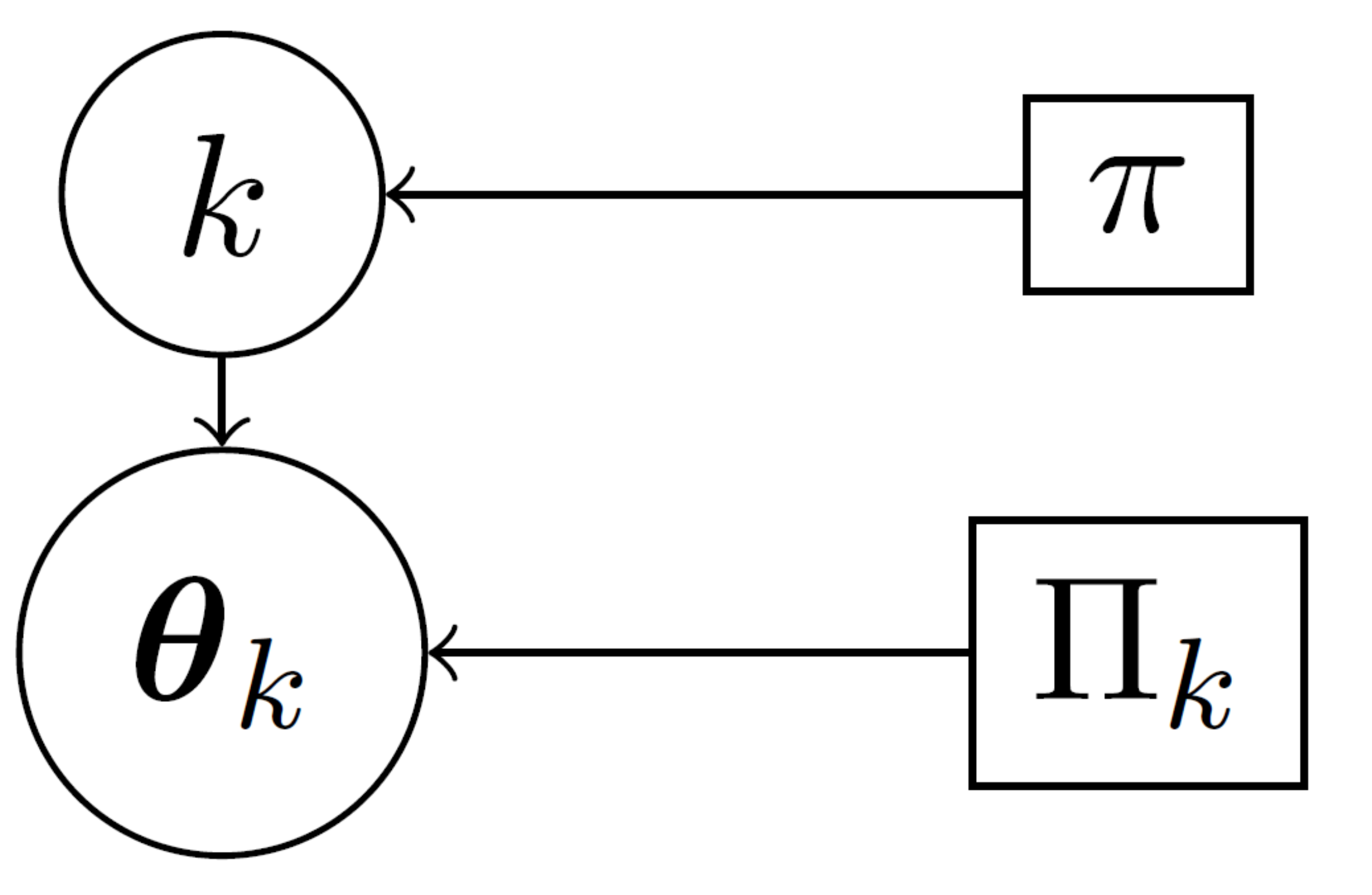}\\
  \caption{Graphical representation of the hierarchical structure of the
\emph{sieve prior} given by Equation (\ref{eq:sievep})}\label{fig:sieve}
\end{center}
\end{figure}

The posterior distribution $\Pi(\,\cdot\,\vert X^{n})$  is defined
by, for all measurable sets $B$ of $\Theta$,
\begin{equation}
\Pi(B|X^{n})=\frac{\int_{B}p_{\boldsymbol{\theta}}^{(n)}(X^{n})d\Pi(\boldsymbol{\theta})}{\int_{\Theta}p_{\boldsymbol{\theta}}^{(n)}(X^{n})d\Pi(\boldsymbol{\theta})}.\label{eq:pos}\end{equation}
Given the sieve prior $\Pi$, we study the rate of contraction of the
posterior distribution in
$P_{\boldsymbol{\theta}_0}^{(n)}-$probability with respect to a
semimetric $d_{n}$ on $\Theta$. This rate is defined as the best
possible (\textit{i.e.} the smallest) sequence
$(\epsilon_{n})_{n\geq1}$ such that
\begin{equation*}
 \Pi\left(\boldsymbol{\theta}:\,
d_{n}^2(\boldsymbol{\theta},\boldsymbol{\theta}_{0})\geq
M\epsilon_{n}^2|X^{n}\right) \underset{n\to\infty}{\longrightarrow}0
,\end{equation*} in $P_{\boldsymbol{\theta}_0}^{(n)}$ probability,
for some $\boldsymbol{\theta}_{0}\in\Theta$ and a positive constant
$M$, which can be chosen as large as needed. We also derive
convergence rates for the posterior loss
$\Pi(d_{n}^{2}(\boldsymbol{\theta},\boldsymbol{\theta}_{0})\vert
X^{n})$ in $P_{\boldsymbol{\theta}_0}^{(n)}-$probability.

The posterior concentration rate is optimal when it coincides with
the minimax rates of convergence, when $\boldsymbol{\theta}_0$
belongs to a given functional class, associated to the same
semimetric $d_n$. Typically these minimax rates of convergence are
defined for functional classes indexed by a smoothness parameter
Sobolev, H\"older, or more generally Besov spaces.

The objective of this paper is to find  mild generic assumptions on
the sieve prior $\Pi$ of the form (\ref{eq:sievep}), on models
$P_{\boldsymbol{\theta}}^{(n)}$ and on  $d_{n}$, such that the
procedure adapts to the optimal rate in the minimax sense, both for
the posterior distribution and for the risk. Results in Bayesian nonparametrics literature about contraction rates are usually of two kinds.
Firstly, general assumptions on priors and models allow to derive rates, see for example \citet{Shen:2001p194,Ghosal:2000p242,Ghosal:2007p267}. Secondly, other papers focus on a particular
prior and obtain contraction rates in a particular model, see for instance \citet{Belitser:2003p168} in the white noise model, \citet{de2010adaptive} in regression, and \citet{scricciolo_convergence_2006} in density.
The novelty of this paper is that our results hold for a family of priors (sieve priors) without a specific underlying model, and can be applied to different models.

An additional interesting property that is sought at the same time
as convergence rates is
 adaptation. This means that, once specified a loss function (a semimetric $d_n$ on $\Theta$), and a collection of classes of
different smoothnesses for the parameter, one constructs a procedure
which is independent of the smoothness, but which is rate optimal
(under the given loss $d_n$), within each class. Indeed, the optimal
rate naturally depends on the smoothness of the parameter, and
standard straightforward estimation techniques usually use it as an
input. This is all the more an important issue that relatively few
instances in the Bayesian literature are available in this area.
That property is often obtained when the unknown parameter is
assumed to belong to a discrete set, see for example
\citet{Belitser:2003p168}. There exist some results in the context
of density estimation by \citet{Huang:2004p7016},
\citet{scricciolo_convergence_2006},
 \citet{Ghosal:2008p6659}, \citet{van_der_vaart_adaptive_2009}, \citet{Rivoirard:2009p6204},  \citet{rousseau2010rates} and \citet{kruijer2010adaptive},
 in regression by
\citet{de2010adaptive}, and in spectral density estimation by
\citet{kruijer_spectral}.
What enables adaptation
in our results is the thresholding induced by the prior on $k$: the
posterior distribution of  parameter $k$ concentrates around values
that are the typical efficient size of models of the true
smoothness.

As seen from our assumptions in Section \ref{sec:assumptions} and
from the general results (Theorem \ref{thm:apo} and Corollary
\ref{cor:point}), adaptation is relatively straightforward under
sieve priors defined by (\ref{eq:sievep}) when the semimetric is a
global loss function which acts like the Kullback-Leibler
divergence, the $l^2$ norm on $\boldsymbol{\theta}$ in the
regression problem, or the Hellinger distance in the density
problem. If the loss function (or the semimetric) $d_n$ acts
differently, then the posterior distribution (or the risk) can be
quite different (suboptimal). This is illustrated in Section 3.2 for
the white noise model (\ref{eq:wnm2}) when the loss is a local loss
function as in the case of the estimation of $\boldsymbol{f}(t)$,
for a given $t$, where $d_n(\boldsymbol{f},\boldsymbol{f}_0) =
(\boldsymbol{f}(t)-\boldsymbol{f}_0(t))^2$. This phenomenon has been
encountered also by \citet{kruijer_spectral}. It  is not merely a
Bayesian issue: \citet{Cai:2007p90} show that an optimal estimator
under global loss cannot be locally optimal at each point
$\boldsymbol{f}(t)$ in the white noise model. The penalty between
global and local rates is at least a $\log n$ term.
\citet{abramovich_optimality_2004} and \citet{Abramovich:2007p6014}
obtain similar results with Bayesian wavelet estimators in the same
model.

The paper is organized as follows. Section \ref{sec:General-case}
first provides a general result on rates of contraction for the
posterior distribution in the setting of sieve priors.
We also derive a result in terms of posterior loss, and show that
the rates are adaptive optimal for Sobolev smoothness classes. The
section ends up with applications to the density, the regression
function and the nonlinear autoregression function estimation. In
Section \ref{sec:app}, we study more precisely the case of the white
noise model, which is a benchmark model. We study in detail the
difference between global or pointwise losses in this model, and
provide a lower bound for the latter loss, showing that sieve priors
lead to suboptimal contraction rates. Proofs are deferred to the
Appendix.

\subsection*{Notations}
We use the following notations. Vectors are written in bold letters,
for example $\boldsymbol{\theta}$ or $\boldsymbol{\theta}_0$, while
light-face is used for their entries, like $\theta_j$ or
$\theta_{0j}$. We denote
by 
$\boldsymbol{\theta}_{0k}$ the projection of
$\boldsymbol{\theta}_{0}$ on its first $k$ coordinates, and by
$p_{0k}^{(n)}$ and $p_{0}^{(n)}$ the densities of the observations
in the corresponding models. We denote by $d_{n}$ a semimetric,
 by $\Vert \cdot \Vert_2 $
 the $l^2$ norm (on vectors) in $\Theta$ or the $L^2$ norm (on curves $\boldsymbol{f}$), and by
$\Vert \cdot\Vert_{2,k}$ the $l^2$ norm restricted to the first $k$
coordinates of a parameter. Expectations $\E_{0}^{(n)}$ and
$\E_{\boldsymbol{\theta}}^{(n)}$ are defined with respect to
$P_{\boldsymbol{\theta}_0}^{(n)}$ and
$P_{\boldsymbol{\theta}}^{(n)}$ respectively. The same notation
$\Pi\left(\,\cdot\,\vert X^n\right)$ is used for posterior
probability or posterior expectation. The expected posterior risk
and the frequentist risk relative to $d_n$ are defined and denoted
by
$\mathcal{R}_n^{d_n}(\boldsymbol{\theta}_{0})=\E_{0}^{(n)}\Pi(d_{n}^{2}(\boldsymbol{\theta},\boldsymbol{\theta}_{0})\vert
X^{n})$ and
$R_n^{d_n}(\boldsymbol{\theta}_{0})=\E_{0}^{(n)}(d_{n}^2(\widehat{\boldsymbol{\theta}},\boldsymbol{\theta}_{0}))$
respectively (for an estimator $\widehat{\boldsymbol{\theta}}$ of
$\boldsymbol{\theta}_0$), where the mention of
$\boldsymbol{\theta}_{0}$ might be omitted \citep[\textit{cf.}][Section 2.3]{robert2007bayesian}.
We denote by $\varphi$ 
the standard Gaussian probability density.

Let $K$ denote the Kullback-Leibler divergence $K(f,g) = \int
f\log(f/g)d\mu$, and $V_{m,0}$ denote the $m^{\text{th}}$ centered
moment $V_{m,0}(f,g) = \int f | \log(f/g) -\allowbreak K(f,g)
|^{m}\allowbreak d\mu$, with $m\geq 2$.

Define two additional divergences $\widetilde{K}$ and
$\widetilde{V}_{m,0}$, which are expectations with respect to
$p_0^{(n)}$, $\widetilde{K}(f,g) = \int p_{0}^{(n)}|\log(f/g)|d\mu$
and $\widetilde{V}_{m,0}(f,g) = \int p_{0}^{(n)}|\log(f/g)
-K(f,g)|^{m}d\mu$.

We denote by $C$ a generic constant whose value is of no importance
and we use $\lesssim$ for inequalities up to a multiple constant.

\section{General case\label{sec:General-case}}

In this section we give a general theorem which provides an upper
bound on posterior contraction rates $\epsilon_n$. Throughout the
section, we assume that the sequence of positive numbers
$\left(\epsilon_{n}\right)_{n\geq1}$, or
$\left(\epsilon_{n}(\beta)\right)_{n\geq1}$ when we point to a
specific value of smoothness $\beta$, is such that
$\epsilon_{n}\underset{n\rightarrow\infty}{\longrightarrow}0$ and
$n\epsilon_{n}^{2}/\log
n\underset{n\rightarrow\infty}{\longrightarrow}\infty$.

We introduce the following numbers
\begin{equation}\label{kn}
j_{n}= \lfloor j_{0}n\epsilon_{n}^{2}/\log(n)\rfloor,\,
k_{n}=\lfloor M_0 j_n \log(n)/L(n) \rfloor, 
\end{equation}
 for $j_{0}>0,M_0>1$, where $L$ is a slow varying function such that $L\leq \log$, hence $j_n\leq k_n$.  We use $k_n$ to define the following approximation
 subsets
of $\Theta$
\begin{align*}
&\Theta_{k_n}(Q) = \left\{
\boldsymbol{\theta}\in\Theta_{k_n}:\,\left\Vert \boldsymbol{\theta}\right\Vert_{2,k_n} \leq
n^Q\right\},
\end{align*}
for $Q>0$. Note that the prior actually charges a union of spaces of
dimension $k$, $k\geq 1$, so that $\Theta_{k_n}(Q)$ can be seen as a
union of spaces of dimension $k\leq k_n$. Lemma \ref{lem:prior_mass}
provides an upper bound on the prior mass of $\Theta_{k_n}(Q)$.

It has been shown
\citep{Ghosal:2000p242,Ghosal:2007p267,Shen:2001p194} that an
efficient way to derive rates of contraction of posterior
distributions is to bound from above the numerator of (\ref{eq:pos})
using tests (and $k_n$ for the increasing
sequence $\Theta_{k_n}(Q)$), and to bound from below its denominator
using an approximation of $p_0^{(n)}$ based on a value
$\boldsymbol{\theta} \in \Theta_{j_n}$ close to $\boldsymbol{\theta}$. The latter is done in Lemma \ref{lem:minoration_bn} where we use $j_n$ to define the finite
component approximation $\boldsymbol{\theta}_{0j_n}$ of
$\boldsymbol{\theta}_{0}$, and we show that the prior mass of the following Kullback-Leibler
neighbourhoods of $\boldsymbol{\theta}_{0}$, $\mathcal{B}_{n}(m)$,
$n\in\mathbb{N}^*$, are lower bounded by an exponential term:
$$\mathcal{B}_{n}(m) = \left\{ \boldsymbol{\theta}:\,
K\left(p_{0}^{(n)},p_{\boldsymbol{\theta}}^{(n)}\right)\leq
2n\epsilon_{n}^{2},V_{m,0}\left(p_{0}^{(n)},p_{\boldsymbol{\theta}}^{(n)}\right)\leq 2^{m+1}(n\epsilon_{n}^{2})^{m/2}\right\}.$$

Define two  neighbourhoods of $\boldsymbol{\theta}_{0}$ in the sieve
space $\Theta_{j_n}$, $\widetilde{\mathcal{B}}_{n}(m)$, similar to
$\mathcal{B}_{n}(m)$ but using $\widetilde{K}$ and
$\widetilde{V}_{m,0}$, and $\mathcal{A}_{n}(H_1)$, an $l^2$ ball of
radius $n^{-H_1}$, $H_1>0$:
\begin{align*}
&\widetilde{\mathcal{B}}_{n}(m) = \left\{ \boldsymbol{\theta}\in\Theta_{j_n}:\,\widetilde{K}\left(p_{0j_n}^{(n)},p_{\boldsymbol{\theta}}^{(n)}\right)\leq n\epsilon_{n}^{2},
\widetilde{V}_{m,0}\left(p_{0j_n}^{(n)},p_{\boldsymbol{\theta}}^{(n)}\right)\leq \left(n\epsilon_{n}^{2}\right)^{m/2}\right\} ,\\
&\mathcal{A}_{n}(H_1) = \left\{ \boldsymbol{\theta}\in\Theta_{j_n}:\,\left\Vert
\boldsymbol{\theta}_{0j_n}-\boldsymbol{\theta}\right\Vert_{2,j_n}\leq
n^{-H_1}\right\}.
\end{align*}

\subsection{Assumptions\label{sec:assumptions}}

The following technical assumptions are involved in the subsequent
analysis, and are discussed at the end of this section. Recall that the true parameter is $\boldsymbol{\theta}_0$, under which the observations have density $p_{0}^{(n)}$.

$\boldsymbol{A_1}$ \textbf{Condition on
$p_{0}^{(n)}$} \textbf{and} $\epsilon_{n}$. For
$n$ large enough and for some $m>0$,
\begin{equation*}
K\left(p_{0}^{(n)},p_{0j_n}^{(n)}\right)\leq n\epsilon_{n}^{2}\quad\text{ and }\quad V_{m,0}\left(p_{0}^{(n)},p_{0j_n}^{(n)}\right)\leq\left(n\epsilon_{n}^{2}\right)^{m/2}.
\end{equation*}

$\boldsymbol{A_2}$ \textbf{Comparison between norms}. The following
inclusion holds in $\Theta_{j_n}$
 \begin{equation*} \exists H_1>0,\text{
s.t.
}\Ac_{n}(H_1)\subset\widetilde{\Bc}_{n}(m).
\end{equation*}

$\boldsymbol{A_3}$ \textbf{Comparison between} $d_{n}$ \textbf{and}
$l^2$. There exist three non negative constants $D_{0},D_{1},D_{2}$
such that, for any two
$\boldsymbol{\theta},\boldsymbol{\theta}'\in\Theta_{k_n}(Q)$,
\begin{equation*}
d_{n}(\boldsymbol{\theta},\boldsymbol{\theta}')\leq D_{0}k_{n}^{D_{1}}\left\Vert \boldsymbol{\theta}-\boldsymbol{\theta}'\right\Vert_{2,k_n} ^{D_{2}}.
\end{equation*}

$\boldsymbol{A_4}$ \textbf{Test Condition}. There exist two positive
constants $c_1$ and $\zeta < 1$ such that, for every
$\boldsymbol{\theta}_{1}\in\Theta_{k_n}(Q)$, there exists a test
$\phi_{n}(\boldsymbol{\theta}_1)\in[0,1]$ which satisfies
\begin{align*}
&\E_{0}^{(n)}(\phi_{n}(\boldsymbol{\theta}_1))\leq
e^{-c_{1}nd_{n}^2(\boldsymbol{\theta}_{0},\boldsymbol{\theta}_{1})}\quad\text{ and
}\quad\\
&\sup_{d_{n}(\boldsymbol{\theta},\boldsymbol{\theta}_{1})<\zeta d_{n}(\boldsymbol{\theta}_{0},\boldsymbol{\theta}_{1})}\E_{\boldsymbol{\theta}}^{(n)}\left(1-\phi_{n}(\boldsymbol{\theta}_1)\right)\leq
e^{-c_{1}nd_{n}^2(\boldsymbol{\theta}_{0},\boldsymbol{\theta}_{1})}.
\end{align*}

$\boldsymbol{A_5}$ \textbf{On the prior} $\Pi$. There exist positive 
constants $a,b,G_1,G_2,G_3,G_4,H_2,\alpha$ and $\tau_0$ such that
$\pi$ satisfy
\begin{equation}
\forall k=1,2,\ldots, \quad e^{-akL(k)}\leq  \pi(k)  \leq e^{-bkL(k)},\label{eq:prior k}
\end{equation}
where the function $L$ is a slow varying function introduced in
Equation (\ref{kn});  $g$ satisfy
{
\begin{align}
\forall {\theta} \in \R,\quad
G_{1}e^{-G_{2}\left|{\theta}\right|^{\alpha}}\leq & g({\theta})  \leq
G_{3}e^{-G_{4}\left|{\theta}\right|^{\alpha}}.\label{eq:priorg}
\end{align}
}
The scales $\boldsymbol{\tau}$ defined in
Equation (\ref{eq:gn}) satisfy the following conditions
\begin{align}
&\max_{j\geq 1}\tau_{j}\leq\tau_{0},\label{eq:maxtau}\\
&\min_{j\leq k_{n}}\tau_{j}\geq n^{-H_2},\label{eq:mintau}\\
&\sum_{j=1}^{j_{n}}\left|{\theta}_{0j}\right|^{\alpha}/\tau_{j}^{\alpha}\leq Cj_{n}\log n.\label{eq:somme theta tau}\end{align}

\begin{rem}\noindent

\begin{itemize}
\item Conditions $\boldsymbol{A_1}$ and $\boldsymbol{A_2}$ are local in that they need to be checked at the true parameter $\boldsymbol{\theta}_0$ only. They are
useful to prove that the prior puts sufficient mass around
Kullback-Leibler neighbourhoods of the true probability.
Condition $\boldsymbol{A_1}$ is a limiting factor to the rate:
it characterizes $\epsilon_{n}$ through the capacity of
approximation of $p_{0}^{(n)}$ by $p_{0j_n}^{(n)}$: the smoother
$p_{0}^{(n)}$, the closer $p_{0}^{(n)}$ and $p_{0j_n}^{(n)}$,
and the faster $\epsilon_{n}$. In many models, they are ensured
because $K(p_{0}^{(n)},p_{\boldsymbol{\theta}_{j_n}}^{(n)})$ and
$V_{m,0}(p_{0}^{(n)},p_{\boldsymbol{\theta}_{j_n}}^{(n)})$ can
be written locally (meaning around $\boldsymbol{\theta}_0$) in
terms of the $l^2$ norm $\Vert
\boldsymbol{\theta}_0-\boldsymbol{\theta}_{j_n}\Vert_2 $
directly. Smoothness assumptions are then typically required to
control $\Vert
\boldsymbol{\theta}_0-\boldsymbol{\theta}_{j_n}\Vert_2 $.

 It is the case for instance for Sobolev and Besov smoothnesses
(\textit{cf.} Equation (\ref{eq:borne sob})). The control is
expressed with a power of $j_n$, whose comparison to
$\epsilon_n^2$ provides in turn a tight way to tune the rate
(\textit{cf.} the proof of Proposition \ref{prop:optimal}).

Note that the constant $H_1$ in Condition $\boldsymbol{A_2}$ can
be chosen as large as needed: if $\boldsymbol{A_2}$ holds for  a
specified positive constant $H_0$, then it does for any
$H_1>H_0$. This makes the condition quite loose. A more
stringent version of $\boldsymbol{A_2}$, if simpler, is the
following.

$\boldsymbol{A_2'}$ \textbf{Comparison between norms}. For any
$\boldsymbol{\theta}\in\Theta_{j_n}$
\begin{eqnarray*}
\widetilde{K}\left(p_{0j_n}^{(n)},p_{\boldsymbol{\theta}}^{(n)}\right) & \leq & Cn\left\Vert \boldsymbol{\theta}_{0j_n}-\boldsymbol{\theta}\right\Vert_{2,j_n} ^{2}\text{ and }\\
\widetilde{V}_{m,0}\left(p_{0j_n}^{(n)},p_{\boldsymbol{\theta}}^{(n)}\right)
 & \leq & Cn^{m/2}\left\Vert
 \boldsymbol{\theta}_{0j_n}-\boldsymbol{\theta}\right\Vert_{2,j_n}
 ^{m}.\end{eqnarray*} This is satisfied in the Gaussian white
noise model (see Section \ref{sec:app}).

\item Condition $\boldsymbol{A_3}$ is generally mild. The reverse
is more stringent since $d_{n}$ may be bounded, as is the case
with the Hellinger distance. $\boldsymbol{A_3}$ is satisfied in
many common situations, see for example the applications later
on. Technically, this condition allows to switch from a covering
number (or entropy) in terms of the $l^2$ norm to a covering
number in terms of the semimetric $d_n$.

\item Condition $\boldsymbol{A_4}$ is common in the Bayesian nonparametric literature. A review of
different models and their corresponding tests is given in
\citet{Ghosal:2007p267} for example. The tests strongly depend
on the semimetric $d_n$.

\item Condition $\boldsymbol{A_5}$ concerns the prior. Equations (\ref{eq:prior k}) and (\ref{eq:priorg}) state that the tails
of $\pi$ and $g$ have to be at least exponential or of
exponential type. For instance, if $\pi$ is the geometric
distribution, $L=1$, and if it is the Poisson distribution,
$L(k)=\log(k)$ (both are slow varying functions). Laplace and
Gaussian distributions are covered by $g$, with  $\alpha=1$ and
$\alpha=2$ respectively. These equations aim at controlling the
prior mass of $\Theta_{k_n}^{c}(Q)$, the complement of
$\Theta_{k_n}(Q)$ in $\Theta$ (see Lemma \ref{lem:prior_mass}).
The case where the scale $\boldsymbol{\tau}$  depends  on $n$ is
considered in \citet{Babenko:2009p2,babenko2010oracle} in the
white noise model. Here the constraints on $\boldsymbol{\tau}$
are rather mild since they are allowed to go to zero polynomially
as a function of $n$, and must be upper bounded.
\citet{Rivoirard:2009p6204} study a family of scales
$\boldsymbol{\tau}=(\tau_{j})_{j\geq 1}$ that are decreasing
polynomially with $j$. Here the prior is more general and
encompasses both frameworks. Equations (\ref{eq:prior k}) -
(\ref{eq:somme theta tau}) are needed in Lemmas
\ref{lem:prior_mass} and \ref{lem:minoration_bn}  for bounding
respectively $\Pi(\mathcal{B}_{n}(m))$ from below and
$\Pi(\Theta_{k_n}^{c}(Q))$ from above. A smoothness assumption
on $\boldsymbol{\theta}_0$ is usually required for Equation
(\ref{eq:somme theta tau}).
\end{itemize}
\end{rem}

\subsection{Results}
\subsubsection{Concentration and posterior loss\label{seq:conc}}
The following theorem provides an upper bound for the rate of
contraction of the posterior distribution.
\begin{thm}
\label{thm:apo}If  Conditions  $\boldsymbol{A_1}$ -
$\boldsymbol{A_5}$ hold, then for $M$ large enough and for $L$ introduced in Equation (\ref{kn}),
\begin{equation*}
 \E_{0}^{(n)}\Pi\left(\boldsymbol{\theta}:\
d_{n}^{2}(\boldsymbol{\theta},\boldsymbol{\theta}_{0})\geq M\frac{\log n}{L(n)}\epsilon_{n}^{2}\vert
X^{n}\right) = \mathcal{O}\left((n\epsilon_{n}^{2})^{-m/2}\right)
\underset{n\to\infty}{\longrightarrow}0.
\end{equation*}
\end{thm}
\begin{proof}
See the Appendix.
\end{proof}
The convergence of the posterior distribution at the rate
$\epsilon_{n}$ implies that the expected posterior risk converges
(at least) at the same rate $\epsilon_{n}$, when $d_n$ is bounded.
\begin{cor}
\label{cor:point}Under the assumptions of Theorem \ref{thm:apo},
with a value of $m$ in Conditions $\boldsymbol{A_1}$ and
$\boldsymbol{A_2}$ such that $(n\epsilon_{n}^{2})^{-m/2} =
\mathcal{O}(\epsilon_{n}^{2})$, and if $d_{n}$ is bounded on
$\Theta$, then the expected posterior risk given
$\boldsymbol{\theta}_{0}$ and $\Pi$ converges at least at the same
rate $\epsilon_{n}$
\[
\mathcal{R}_{n}^{d_{n}}=\E_{0}^{(n)}\Pi(d_{n}^{2}(\boldsymbol{\theta},\boldsymbol{\theta}_{0})\vert X^{n})=\go\left(\frac{\log n}{L(n)}\epsilon_{n}^{2}\right).\]
\end{cor}
\begin{proof}
Denote $D$ the bound of $d_{n}$, \textit{i.e.} for all
$\boldsymbol{\theta},\ \boldsymbol{\theta}'\in\Theta$,
$d_{n}(\boldsymbol{\theta},\boldsymbol{\theta}')\leq D$. We have
\begin{eqnarray*}
\mathcal{R}_{n}^{d_{n}} & \leq & M\frac{\log n}{L(n)}\epsilon_{n}^{2}+\E_{0}^{(n)}\Pi\left(\ind\left(d_{n}^{2}(\boldsymbol{\theta},\boldsymbol{\theta}_{0})\geq M\frac{\log n}{L(n)}\epsilon_{n}^{2}\right)d_{n}^{2}(\boldsymbol{\theta},\boldsymbol{\theta}_{0})\vert X^{n}\right)\\
 & \leq & M\frac{\log n}{L(n)}\epsilon_{n}^{2}+D\E_{0}^{(n)}\Pi\left(\boldsymbol{\theta}:\ d_{n}^{2}(\boldsymbol{\theta},\boldsymbol{\theta}_{0})\geq
 M\frac{\log n}{L(n)}\epsilon_{n}^{2}\vert X^{n}\right)\end{eqnarray*}
so $\mathcal{R}_{n}^{d_{n}} = \mathcal{O}(\log n/L(n)\epsilon_n^2)$ by Theorem \ref{thm:apo} and the assumption on $m$.\end{proof}
\begin{rem}
The condition on $m$ in Corollary \ref{cor:point} requires
$n\epsilon_n^2$ to grow as a power of $n$. When
$\boldsymbol{\theta}_{0}$ has Sobolev smoothness $\beta$, this is
the case since $\epsilon_n^2$ is typically of order $(n/\log
n)^{-2\beta/(2\beta+1)}$. The condition on $m$ boils down to
$m\geq4\beta$. When $\boldsymbol{\theta}_{0}$ is smoother,
\textit{e.g.} in a Sobolev space with exponential weights, the rate
is typically of order $\log n/\sqrt{n}$. Then a common way to
proceed is to resort to an exponential inequality for controlling
the denominator of the posterior distribution of Equation
(\ref{eq:pos}) \citep[see \textit{e.g.}][]{Rivoirard:2011}.
\end{rem}
\begin{rem} We can note that this result is meaningful from a non Bayesian point
of view as well. Indeed, let $\widehat{\boldsymbol{\theta}}$ be the
posterior mean estimate of $\boldsymbol{\theta}$ with respect to
$\Pi$. Then, if $\boldsymbol{\theta}\rightarrow
d_{n}^2\left(\boldsymbol{\theta},\boldsymbol{\theta}_{0}\right)$ is
convex, we have by Jensen's inequality\[
d_{n}^2(\widehat{\boldsymbol{\theta}},\boldsymbol{\theta}_{0})\leq\Pi\left(d_n^2(\boldsymbol{\theta},\boldsymbol{\theta}_{0})\vert
X^{n}\right),\] so the frequentist risk converges at the same rate
$\epsilon_{n}$
$$R_{n}^{d_{n}}=\E_{0}^{(n)}(d_{n}^2(\widehat{\boldsymbol{\theta}},\boldsymbol{\theta}_{0}))\leq
\E_{0}^{(n)}\Pi\left(d_{n}^{2}(\boldsymbol{\theta},\boldsymbol{\theta}_{0})\vert X^{n}\right)=\mathcal{R}_{n}^{d_{n}}=\go\left(\frac{\log n}{L(n)}\epsilon_{n}^{2}\right).$$
Note that we have no result for general pointwise estimates $\widehat{\boldsymbol{\theta}}$, for instance for the MAP. This latter was studied in \citet{abramovich2007optimality,abramovich2010bayesian}.
\end{rem}

\subsubsection{Adaptation\label{seq:ad}}
When considering a given class of smoothness for the parameter
$\boldsymbol{\theta}_{0}$, the minimax criterion implies an optimal
rate of convergence. Posterior (resp. risk) adaptation means that
the posterior distribution (resp. the risk) concentrates at the
optimal rate for a class of possible smoothness values.

We consider here Sobolev classes $\Theta_{\beta}(L_0)$ {for univariate problems} defined by
\begin{equation} \Theta_{\beta}(L_0)=\left\{
\boldsymbol{\theta}:\sum_{j=1}^{\infty}{\theta}_{j}^{2}j^{2\beta}<L_0\right\},\,\beta >1/2,\,L_0>0\label{eq:sob}\end{equation}
with
smoothness parameter $\beta$ and radius $L_0$. If $\boldsymbol{\theta}_0\in
\Theta_{\beta}(L_0)$, then one has the following bound
\begin{equation} \left\Vert
\boldsymbol{\theta}_{0}-\boldsymbol\theta_{0j_{n}}\right\Vert_2
^{2}=\sum_{j=j_{n}+1}^{\infty}\theta_{0j}^{2}j^{2\beta}j^{-2\beta}\leq
L_{0}j_{n}^{-2\beta}.\label{eq:borne sob}\end{equation}
\citet{Donoho:1998p1331} give the global (\textit{i.e.} under the
$l^2$ loss) minimax rate $n^{-\beta/(2\beta+1)}$ attached to
the Sobolev class of smoothness $\beta$. We show that under an
additional condition between $K$, $V_{m,0}$ and $l^2$, the upper
bound $\epsilon_{n}$ on the rate of contraction can be chosen equal
to the optimal rate, up to a $\log n$ term.
\begin{prop}
\label{prop:optimal} Let $L_0$ denote a positive fixed radius, and
$\beta_2\geq  \beta_1>1/2$. If for $n$ large enough, there exists a
positive constant $C_0$ such that
\begin{align}
&\sup_{\beta_1\leq\beta\leq\beta_2}\sup_{\boldsymbol{\theta}_0\in \Theta_\beta(L_0)}K\left(p_{0}^{(n)},p_{0j_{n}}^{(n)}\right)\leq
C_0n\left\Vert \boldsymbol{\theta}_{0}-\boldsymbol{\theta}_{0j_{n}}\right\Vert_2
^{2},\,\text{and}\nonumber\\
&\sup_{\beta_1\leq\beta\leq\beta_2}\sup_{\boldsymbol{\theta}_0\in
\Theta_\beta(L_0)}V_{m,0}\left(p_{0}^{(n)},p_{0j_{n}}^{(n)}\right)\leq
C_0^mn^{m/2}\left\Vert
\boldsymbol{\theta}_{0}-\boldsymbol{\theta}_{0j_{n}}\right\Vert_2^{m},\label{eq:optimality}\end{align} and if Conditions
$\boldsymbol{A_2}$ - $\boldsymbol{A_5}$ hold with constants
independent of $\boldsymbol{\theta}_{0}$ in the set
$\cup_{\beta_1\leq\beta\leq\beta_2} \Theta_\beta(L_0)$,
 then for $M$ sufficiently
large,
\begin{equation*}
\sup_{\beta_1\leq\beta\leq\beta_2}\sup_{\boldsymbol{\theta}_0\in \Theta_\beta(L_0)}
 \E_{0}^{(n)}\Pi\left(\boldsymbol{\theta}:\
d_{n}^{2}(\boldsymbol{\theta},\boldsymbol{\theta}_{0})\geq M\frac{\log n}{L(n)}\epsilon_{n}^{2}(\beta)\vert
X^{n}\right)\underset{n\to\infty}{\longrightarrow}0,
\end{equation*}
with
\[
\epsilon_{n}(\beta)=\epsilon_{0}\left(\frac{\log
n}{n}\right)^{\frac{\beta}{2\beta+1}},\] and $\epsilon_0$ depending
on $L_0,C_0$ and the constants involved in the assumptions, but
not depending on $\beta$.
\end{prop}
\begin{rem}
In the standard case where $d_{n}$ is the $l^2$ norm, $\epsilon_{n}$
is the \emph{optimal rate of contraction}, up to a $\log n$ term
(which is quite common in Bayesian nonparametric computations).
\end{rem}
\begin{proof}
Let $\beta\in[\beta_1,\beta_2]$ and $\boldsymbol{\theta}_0\in
\Theta_\beta(L_0)$. Then  $\boldsymbol{\theta}_0$ satisfies Equation
(\ref{eq:borne sob}), and Condition (\ref{eq:optimality}) implies
that\[ K\left(p_{0}^{(n)},p_{0j_{n}}^{(n)}\right)\leq
C_0L_{0}nj_{n}^{-2\beta},\
V_{m,0}\left(p_{0}^{(n)},p_{0j_{n}}^{(n)}\right)\leq
C_0L_{0}^{m}n^{m/2}j_{n}^{-m\beta}.\] 

For given $\boldsymbol{\theta}_0$ and $\beta$, the result of Theorem \ref{thm:apo} holds if Condition $\boldsymbol{A_1}$
is satisfied. This is the case if we choose
$\epsilon_{n}(\beta,\boldsymbol{\theta}_0)\geq C_0L_0 j_{n}^{-\beta}$, provided that the bounds in Conditions $\boldsymbol{A_2}$ - $\boldsymbol{A_5}$ and in Equation (\ref{eq:optimality}) are uniform. Combined with $j_{n}=\lfloor j_{0}n\epsilon_{n}^{2}/\log n\rfloor$,
it gives as a tight choice
$\epsilon_{n}(\beta,\boldsymbol{\theta}_0) =
\epsilon_{0}(\beta,\boldsymbol{\theta}_0) (\log
n/n)^{\beta/(2\beta+1)}$ with $\epsilon_{0}(\beta,\boldsymbol{\theta}_0) \leq (L_0C_0j_0^{-\beta})^{1/(2\beta+1)}$. 
So there exists a bound $\epsilon_0>0$ such
that
$\sup_{\beta_1\leq\beta\leq\beta_2}\sup_{\boldsymbol{\theta}_0\in
\Theta_\beta(L_0)}\epsilon_{0}(\beta,\boldsymbol{\theta}_0)=\epsilon_0<\infty$,
which concludes the proof.
\end{proof}

\subsection{Examples\label{sec:examples}}
In this section, we apply our results of contraction of Sections \ref{seq:conc} and
\ref{seq:ad} to a series of models. The Gaussian white noise example
is studied in detail in Section \ref{sec:app}. We suppose in each model that
$\boldsymbol{\theta}_0\in\Theta_{\beta}(L_0)$, where $\Theta_{\beta}(L_0)$ is defined in Equation (\ref{eq:sob}).

Throughout, we consider the following prior $\Pi$ on $\Theta$ (or on
a curve space $\mathcal{F}$ through the coefficients of the
functions in a basis).
 Let the prior distribution $\pi$ on $k$ be Poisson with parameter
 $\lambda$, and given $k$, the prior distribution on
 $\theta_j/\tau_j$, $j=1,\ldots,k$ be standard Gaussian,
\begin{align}k&\sim  \text{Poisson}(\lambda),\nonumber\\
\frac{\theta_j}{\tau_j}\,|\,k&\sim
\mathcal{N}(0,1),\,j=1,\ldots,k,\,\text{independently}.\label{eq:priorexample}\end{align}
It satisfies Equation (\ref{eq:prior k}) with function
$L(k)=\log(k)$ and Equation (\ref{eq:priorg}) with $\alpha=2$.
Choose then $\tau_{j}^{2}=\tau_0j^{-2q}$, $\tau_0>0$, with
$q>1/2$. It is decreasing and bounded from above by $\tau_0$ so
Equation (\ref{eq:maxtau}) is satisfied. Additionally,
\begin{equation*}\min_{j\leq k_{n}}\tau_{j}=\tau_{k_n}=k_n^{-2q}\geq n^{-H_2}\end{equation*}
for $H_2$ large enough, so Equation (\ref{eq:mintau}) is checked.
Since $\boldsymbol{\theta}_0\in\Theta_{\beta}(L_0)$,
\begin{equation*}
\tau_0^2\sum_{j=1}^{j_{n}}\theta_{0j}^2/\tau_{j}^{2}
=
\sum_{j=1}^{j_{n}}\theta_{0j}^2j^{2q}=\sum_{j=1}^{j_{n}}\theta_{0j}^2j^{2\beta}j^{2q-2\beta}
\leq j_n\sum_{j=1}^{j_{n}}\theta_{0j}^2j^{2\beta}
\leq j_nL_0,
\end{equation*}
as soon as $2q-2\beta\leq 1$. Hence by choosing $ 1/2<q\leq 1$,
Equation (\ref{eq:somme theta tau}) is verified for all $\beta >
1/2$. The prior $\Pi$ thus satisfies Condition $\boldsymbol{A_5}$.

Since Condition $\boldsymbol{A_5}$ is satisfied, we will show in the three examples that Conditions
$\boldsymbol{A_2}$ - $\boldsymbol{A_4}$ and Condition
(\ref{eq:optimality}) hold, thus Proposition \ref{prop:optimal}
applies: the posterior distribution attains the optimal rate of
contraction, up to a $\log n$ term, that is
$\epsilon_{n}=\epsilon_{0}(\log
n/n)^{\beta/(2\beta+1)}$, for a distance $d_n$ which is
specific to each model. This rate is adaptive in a range of smoothness $[\beta_1,\beta_2]$.

\subsubsection{Density}
Let us consider the density model in which the density
$\boldsymbol{p}$ is unknown, and we observe i.i.d. data
\begin{equation*}X_i\sim \boldsymbol{p},\quad
i=1,2,\ldots,n,\end{equation*}
where $\boldsymbol{p}$ belongs to  $\mathcal{F}$,\[
\mathcal{F}=\left\{ \boldsymbol{p}\text{ density on
}[0,1]:\,\boldsymbol{p}(0)=\boldsymbol{p}(1)\text{ and }\log
\boldsymbol{p}\in L^2(0,1)\right\}.\] Equality
$\boldsymbol{p}(0)=\boldsymbol{p}(1)$ is mainly used for ease of
computation. We define the parameter $\boldsymbol{\theta}$ of such a
function $\boldsymbol{p}$, and write
$\boldsymbol{p}=\boldsymbol{p}_{\boldsymbol{\theta}}$, as the
coefficients of $\log\boldsymbol{p}_{\boldsymbol{\theta}}$ in the
Fourier basis $\boldsymbol{\psi}=(\psi_{j})_{j\geq 1}$,
\textit{i.e.} it can be represented as
$$\log\boldsymbol{p}_{\boldsymbol{\theta}}(x)=\sum_{j=1}^\infty
\theta_j\psi_j(x) - c(\boldsymbol{\theta}),$$ where
$c(\boldsymbol{\theta})$ is a normalizing constant. We assign a
prior to $\boldsymbol{p}_{\boldsymbol{\theta}}$ by assigning the
sieve prior $\Pi$ of Equation (\ref{eq:priorexample}) to
$\boldsymbol{\theta}$.

A natural choice of  metric $d_{n}$ in this model is the Hellinger
distance
$d_{n}(\boldsymbol{\theta},\boldsymbol{\theta}')=h(\boldsymbol{p}_{\boldsymbol{\theta}},\boldsymbol{p}_{\boldsymbol{\theta}'})=\left(\int\left(\sqrt{\boldsymbol{p}_{\boldsymbol{\theta}}}-\sqrt{\boldsymbol{p}_{\boldsymbol{\theta}'}}\right)^{2}d\mu\right)^{1/2}$.
Lemma 2 in \citet{Ghosal:2007p267} shows the existence of tests
satisfying $\boldsymbol{A_4}$ with the Hellinger distance.

\citet{Rivoirard:2011} study this model in detail (Section 4.2.2) in
order to derive a Bernstein-von Mises theorem for the density model.
They prove that Conditions $\boldsymbol{A_2}$, $\boldsymbol{A_3}$
and (\ref{eq:optimality}) are valid in this model  (see the proof of
Condition (C) for $\boldsymbol{A_2}$ and (\ref{eq:optimality}), and
the proof of Condition (B) for $\boldsymbol{A_3}$). With
$D_{1}=D_{2}=1$, Condition $\boldsymbol{A_3}$
 is written $h(\boldsymbol{p}_{\boldsymbol{\theta}},\boldsymbol{p}_{\boldsymbol{\theta}'})\leq
D_0k_{n}\left\Vert
\boldsymbol{\theta}-\boldsymbol{\theta}'\right\Vert_{2,k_n}.$

\subsubsection{Regression}
Consider now the following nonparametric regression model\[
X_{i}=\boldsymbol{f}(t_{i})+\sigma\xi_{i},\quad i=1,\ldots, n,\]
with the regular fixed design $t_{i}=i/n$ in $[0,1]$, i.i.d.
centered Gaussian errors $\xi_{i}$ with variance  $\sigma^2$.
{
The unknown $\sigma$ case is studied in an unpublished paper by Rousseau and Sun. They endow $\sigma$ with an Inverse Gamma (conjugate) prior. They show that this one dimensional parameter adds an $n\log(\sigma/\sigma_0)$ term in the Kullback-Leibler divergence but does not alter the rates by considering three different cases for $\sigma$, either $\sigma < \sigma_0/2$, $\sigma > 3\sigma_0/2$, or $\sigma \in [\sigma_0/2,3\sigma_0/2]$.}

We consider now in more detail the $\sigma$ known case.
Denote $\boldsymbol{\theta}$ the coefficients of a regression function
$\boldsymbol{f}$  in the Fourier basis
$\boldsymbol{\psi}=(\psi_{j})_{j\geq 1}$. So for all $t\in[0,1]$,
$\boldsymbol{f}$ can be represented as
$\boldsymbol{f}(t)=\sum_{j=1}^\infty \theta_j\psi_j(t)$. We assign a
prior to $\boldsymbol{f}$ by assigning the sieve prior $\Pi$ of
Equation (\ref{eq:priorexample}) to $\boldsymbol{\theta}$.

Let $\mathbb{P}_t^n=n^{-1}\sum_{i=1}^n\delta_{t_{i}}$ be the
empirical measure of the covariates $t_{i}$'s, and define the square
of the empirical norm by
$\Vert\boldsymbol{f}\Vert_{\mathbb{P}_t^n}^2=n^{-1} \sum_{i=1}^n
\boldsymbol{f}^2(t_{i})$. We use
$d_n=\Vert\,\cdot\,\Vert_{\mathbb{P}_t^n}$.

Let $\boldsymbol{\theta}\in\Theta$ and $\boldsymbol{f}$ the
corresponding regression. Basic algebra \citep[see for example Lemma
1.7 in][]{tsybakov2009introduction} provides, for any two $j$ and
$k$,
$$\frac{1}{n}\sum_{i=1}^n\psi_j(t_{i})\psi_k(t_{i})=\delta_{jk},$$
where  $\delta_{jk}$ stands for Kronecker delta. Hence
\begin{equation}\label{eq:approxemp}\Vert\boldsymbol{f}\Vert_{\mathbb{P}_t^n}^2=\frac{1}{n}\sum_{i=1}^n\sum_{j,k}\theta_j\theta_k\psi_j(t_{i})\psi_k(t_{i})=
\Vert\boldsymbol{\theta}\Vert_2^2=\Vert\boldsymbol{f}\Vert_2^2,
\end{equation}
where the last equality is Parseval's. It ensures Condition
$\boldsymbol{A_3}$ with $D_0=D_2=1$ and $D_1=0$.

The densities $\mathcal{N}(\boldsymbol{f}(t_{i}),\sigma^2)$ of
$X_i$'s are denoted $p_{\boldsymbol{f},i}$, $i=1,\ldots,n$, and
their product $p_{\boldsymbol{f}}^{(n)}$. The quantity $\boldsymbol{f}_{0j_n}$
denotes the truncated version of $\boldsymbol{f}_0$ to its first
$j_n$ terms in the Fourier basis.

 We have
$2K(p_{\boldsymbol{f}_0,i},p_{\boldsymbol{f},i}) =
V_{2,0}(p_{\boldsymbol{f}_0,i},p_{\boldsymbol{f},i}) =
\sigma^{-2}(\boldsymbol{f}_0(t_{i}) - \boldsymbol{f}(t_{i}))^2$ and
$V_{m,0}(p_{\boldsymbol{f}_0,i},\allowbreak p_{\boldsymbol{f},i}) =
\allowbreak \sigma_m\sigma^{m-2}\vert\boldsymbol{f}_0(t_{i}) -
\boldsymbol{f}(t_{i})\vert^2$ for $m\geq 2$, where $\sigma_m$ is the
(non centered) $m^{\text{th}}-$moment of a standard Gaussian
variable. So using Equation (\ref{eq:approxemp}) we get
$$2K(p_{\boldsymbol{f}_0}^{(n)},p_{\boldsymbol{f}}^{(n)}) =
V_{2,0}(p_{\boldsymbol{f}_0}^{(n)},p_{\boldsymbol{f}}^{(n)}) =
n\sigma^{-2}\Vert\boldsymbol{f}_0 -
\boldsymbol{f}\Vert_{\mathbb{P}_t^n}^2 = n\sigma^{-2} \Vert
\boldsymbol{\theta}_{0}-\boldsymbol\theta\Vert_2^2  $$ which proves Condition (\ref{eq:optimality}).

Additionally, both
$2\widetilde{K}(p_{\boldsymbol{f}_{0j_{n}}}^{(n)},p_{\boldsymbol{f}}^{(n)})$
and
$\widetilde{V}_{2,0}(p_{\boldsymbol{f}_{0j_{n}}}^{(n)},p_{\boldsymbol{f}}^{(n)})$
are upper bounded by $n\sigma^{-2}(2\Vert\boldsymbol{f}_{0j_{n}} -
\boldsymbol{f}\Vert_{\mathbb{P}_t^n}^2+\Vert\boldsymbol{f}_0 -
\boldsymbol{f}_{0j_{n}}\Vert_{\mathbb{P}_t^n}^2)$. Let
$\boldsymbol{\theta} \in \mathcal{A}_{n}(H_1)$, for a certain
$H_1>0$. Then, using (\ref{eq:approxemp}) again, the bound is less
than
$$n\sigma^{-2}(n^{-H_1} + L_0j_{n}^{-2\beta})\leq Cn\epsilon_n^2$$
for  $H_1>2\beta/(2\beta+1)$, which ensures Condition
$\boldsymbol{A_2}$.

\citet{Ghosal:2007p267} state in Section 7.7 that tests satisfying
$\boldsymbol{A_4}$ exist with
$d_n=\Vert\,\cdot\,\Vert_{\mathbb{P}_t^n}$.

\subsubsection{Nonlinear AR(1) model\label{AR1}}
As a nonindependent illustration, we consider the following  Markov
chain: the nonlinear autoregression model whose observations
$X^n=(X_1,\ldots,X_n)$ come from a stationary time series ${X_t,t\in
\mathbb{Z}}$, such that
$$X_i=\boldsymbol{f}(X_{i-1})+\xi_i,\quad i=1,2,\ldots,n,$$
where the function $\boldsymbol{f}$ is unknown and the residuals
$\xi_i$ are standard Gaussian and independent of
$(X_1,\ldots,X_{i-1})$. We suppose that $X_0$ is drawn in the
stationary distribution.

Suppose that regression functions $\boldsymbol{f}$ are in $L_2(\R)$,
and uniformly bounded by a constant $M_1$ (a bound growing with $n$
could also be considered here). We use Hermite functions
$\boldsymbol{\psi}=(\psi_{j})_{j\geq 1}$ as an orthonormal basis of
$\R$, such that for all $x\in\R$,
$\boldsymbol{f}(x)=\boldsymbol{f}_{\boldsymbol{\theta}}(x)=\sum_{j=1}^\infty
\theta_j\psi_j(x)$. This basis is uniformly bounded (by Cram\'er's
inequality). {Consider the sieve prior $\Pi$ in its truncated version
(\ref{eq:gntrunc}) for $\boldsymbol{\theta}$, with radius $r_1$ a (possibly large) constant independent of $k$ and $n$.}

We show that Conditions $\boldsymbol{A_1}$-$\boldsymbol{A_4}$ are
satisfied, along the lines of \citet{Ghosal:2007p267} Sections 4 and
7.4. Denote
$p_{\boldsymbol{\theta}}(y|x)=\varphi(y-\boldsymbol{f}_{\boldsymbol{\theta}}(x))$
the transition density of the chain, where $\varphi(\,\cdot\,)$ is
the standard normal density distribution, and where reference
measures relative to $x$ and $y$ are denoted respectively by $\nu$
and $\mu$. Define $r(y)=\frac{1}{2}(\varphi(y-M_1)+\varphi(y+M_1))$,
and set $d\nu=rd\mu$. Then \citet{Ghosal:2007p267} show that the
chain $(X_i)_{1\leq i \leq n}$ has a unique stationary distribution
and prove the existence of tests satisfying $\boldsymbol{A_4}$
relative to the Hellinger semidistance $d$ whose square is given by
$$d^2(\boldsymbol{\theta},\boldsymbol{\theta'}) = \int\int\left(\sqrt{p_{\boldsymbol{\theta}}(y|x)}-\sqrt{p_{\boldsymbol{\theta'}}(y|x)}\right)^2d\mu(y)d\nu(x).$$
They show that $d$ is bounded by
$\Vert\,\cdot\,\Vert_2$ (which proves Condition $\boldsymbol{A_3}$)
and that
$$K(p_0,p_{\boldsymbol{\theta}}) =
V_{2,0}(p_0,p_{\boldsymbol{\theta}}) \lesssim
\Vert\boldsymbol\theta_0 - \boldsymbol{\theta}\Vert_2^2.
$$
Thus Equation (\ref{eq:optimality}) holds. Condition
$\boldsymbol{A_2}$ follows from  inequalities
$\widetilde{K}(p_{0j_{n}},p_{\boldsymbol{\theta}})\lesssim
\sum_{j=1}^{j_{n}}|\theta_{0j}-\theta_j|$ and
$\widetilde{V}_{2,0}(p_{0j_{n}},p_{\boldsymbol{\theta}})\lesssim
\Vert\boldsymbol\theta_{0j_{n}} -
\boldsymbol{\theta}\Vert_{2,j_{n}}^2$ for
$\boldsymbol{\theta}\in\Theta_{j_{n}}$.

\section{Application to the white noise model\label{sec:app}}
Consider the Gaussian white noise model
\begin{equation}
dX^{n}(t)=\boldsymbol{f}_{0}(t)dt+\frac{1}{\sqrt{n}}dW(t),\quad0\leq
t\leq1,\label{eq:wnm2}\end{equation} in which we observe processes $X^{n}(t)$, where
$\boldsymbol{f}_{0}$ is the unknown function of interest belonging
to $L^2(0,1)$, $W(t)$ is a standard Brownian motion, and $n$ is the
sample size. We assume that $\boldsymbol{f}_{0}$ lies in a Sobolev
ball, $\Theta_{\beta}(L_0)$, see (\ref{eq:sob}).
 \citet{Brown:1996p21} show that this model is asymptotically equivalent
to the nonparametric regression (assuming $\beta>1/2$). It can be
written as the equivalent infinite normal mean model using the
decomposition in a Fourier basis
$\boldsymbol{\psi}=(\psi_{j})_{j\geq 1}$ of $L^2(0,1)$,
\begin{equation}
X_{j}^{n}=\theta_{0j}+\frac{1}{\sqrt{n}}\xi_{j},\, j=1,2,\ldots\label{eq:wnm}
\end{equation}
where $X_{j}^{n}=\int_{0}^{1}\psi_{j}(t)\, dX^{n}(t)$ are the
observations, $\theta_{0j}=\int_{0}^{1}\psi_{j}(t)\,
\boldsymbol{f}_{0}(t)dt$ the Fourier coefficients of
$\boldsymbol{f}_{0}$, and $\xi_{j}=\int_{0}^{1}\psi_{j}(t)dW(t)$ are
independent standard Gaussian random variables. The function
$\boldsymbol{f}_{0}$ and the parameter $\boldsymbol{\theta}_{0}$ are
linked through the relation in $L^2(0,1)$,
$\boldsymbol{f}_{0}=\sum_{j=1}^{\infty}\theta_{0j}\psi_{j}$.

In addition to results in concentration, we are interested in
comparing the risk of an estimate $\hat{\boldsymbol{f}}_{n}$
corresponding to basis coefficients $\hat{\boldsymbol{\theta}}_{n}$,
under two different losses: the global $L^2$ loss (if expressed on
curves $\boldsymbol{f}$, or $l^2$ loss if expressed on
$\boldsymbol{\theta}$),
\begin{eqnarray*} R_{n}^{L^2}(\boldsymbol{\theta}_{0}) =
\E_{0}^{(n)}\left\Vert \hat{\boldsymbol{f}}_{n}-\boldsymbol{f}_{0}\right\Vert_2
^{2}=\E_{0}^{(n)}\sum_{j=1}^{\infty}\left(\hat{\theta}_{nj}-\theta_{0j}\right)^{2},\end{eqnarray*}
and the local loss at point $t\in[0,1]$,
\begin{eqnarray*}
R_{n}^{\text{loc}}(\boldsymbol{\theta}_{0},t)=\E_{0}^{(n)}\left(\hat{\boldsymbol{f}}_{n}(t)-\boldsymbol{f}_{0}(t)\right)^{2}=\E_{0}^{(n)}\left(\sum_{j=1}^{\infty}a_{j}\left(\hat{\theta}_{nj}-\theta_{0j}\right)\right)^{2},\end{eqnarray*}
with $a_{j}=\psi_{j}(t)$. Note that the difference between global
and local risks expressions in basis coefficients comes from the
parenthesis position with respect to the square: respectively the
sum of squares and the square of a sum.

We show that sieve priors allow to construct adaptive estimate in
global risk. However, the same estimate does not perform as well
under the pointwise loss, which illustrates the result of
\citet{Cai:2007p90}. We provide a lower bound for the pointwise
rate.

\subsection{Adaptation under global loss}
Consider the global $l^2$ loss on $\boldsymbol{\theta}_{0}$. The
likelihood ratio is given by
\[
\frac{p_{0}^{(n)}}{p_{\boldsymbol{\theta}}^{(n)}}(X^n)=\exp\left(n\langle\boldsymbol{\theta}_{0}-\boldsymbol{\theta},X^{n}\rangle-\frac{n}{2}\left\Vert
\boldsymbol{\theta}_{0}\right\Vert_2 ^{2}+\frac{n}{2}\left\Vert
\boldsymbol{\theta}\right\Vert_2 ^{2}\right),\] where
$\langle.,.\rangle$ denotes the $l^2$ scalar product. We choose here
the $l^2$ distance as
$d_{n}(\boldsymbol{\theta},\boldsymbol{\theta}')=\left\Vert
\boldsymbol{\theta}-\boldsymbol{\theta}'\right\Vert_2 $. Let us check that Conditions
$\boldsymbol{A_2}$ - $\boldsymbol{A_4}$ and Condition
(\ref{eq:optimality}) hold.

The choice of $d_n$ ensures Condition $\boldsymbol{A_3}$ with $D_{0}=D_2=1$ and
$D_{1}=0$. The test statistic of $\boldsymbol{\theta}_{0}$ against
$\boldsymbol{\theta}_{1}$ associated with the likelihood ratio is
$\phi_{n}(\boldsymbol{\theta}_{1})=\ind(2\langle\boldsymbol{\theta}_{1}-\boldsymbol{\theta}_{0},X^{n}\rangle>\left\Vert
\boldsymbol{\theta}_{1}\right\Vert_2 ^{2}-\left\Vert
\boldsymbol{\theta}_{0}\right\Vert_2 ^{2})$. With Lemma 5 of
\citet{Ghosal:2007p267} we have that
$\E_{0}^{(n)}(\phi_{n}(\boldsymbol{\theta}_{1}))\leq e^{-n\left\Vert
\boldsymbol{\theta}_{1}-\boldsymbol{\theta}_{0}\right\Vert_2
^{2}/4}$ and
$\E_{\boldsymbol{\theta}}^{(n)}\left(1-\phi_{n}(\boldsymbol{\theta}_{1})\right)\leq
e^{-n\left\Vert
\boldsymbol{\theta}_{1}-\boldsymbol{\theta}_{0}\right\Vert_2
^{2}/4}$ for $\boldsymbol{\theta}$ such that $\left\Vert
\boldsymbol{\theta}-\boldsymbol{\theta}_{1}\right\Vert_2
\leq\left\Vert
\boldsymbol{\theta}_{1}-\boldsymbol{\theta}_{0}\right\Vert_2 /4$. It
provides a test as in Condition $\boldsymbol{A_4}$ with
$c_{1}=\zeta=1/4$.

Moreover, following Lemma 6 of \citet{Ghosal:2007p267} we have
\begin{equation*}
K\left(p_{0}^{(n)},p_{\boldsymbol{\theta}}^{(n)}\right)=n\left\Vert
\boldsymbol{\theta}-\boldsymbol{\theta}_{0}\right\Vert_2 ^{2}/2\text{ and
}V_{2,0}\left(p_{0}^{(n)},p_{\boldsymbol{\theta}}^{(n)}\right)=n\left\Vert
\boldsymbol{\theta}-\boldsymbol{\theta}_{0}\right\Vert_2 ^{2}.\end{equation*}
For $m\geq2$, we have\begin{eqnarray*}
V_{m,0}\left(p_{0}^{(n)},p_{\boldsymbol{\theta}}^{(n)}\right) & = & \int p_{0}^{(n)}\left|\log\left(p_{0}^{(n)}/p_{\boldsymbol{\theta}}^{(n)}\right)-K\left(p_{0}^{(n)},p_{\boldsymbol{\theta}}^{(n)}\right)\right|^{m}d\mu\\
 & = & n^{m}\int p_{0}^{(n)}\left|\langle\boldsymbol{\theta}_{0}-\boldsymbol{\theta},X^{n}-\boldsymbol{\theta}_{0}\rangle\right|^{m}d\mu\\
 & \leq & n^{m}\left\Vert \boldsymbol{\theta}_{0}-\boldsymbol{\theta}\right\Vert_2 ^{m}\int p_{0}^{(n)}\left\Vert X^{n}-\boldsymbol{\theta}_{0}\right\Vert_2 ^{m}d\mu.\end{eqnarray*}
The centered $m^{\text{th}}-$moment of the Gaussian variable $X^{n}$
is proportional to $n^{-m/2}$, so
$V_{m,0}\left(p_{0}^{(n)},p_{\boldsymbol{\theta}}^{(n)}\right)\lesssim
n^{m/2}\left\Vert
\boldsymbol{\theta}_{0}-\boldsymbol{\theta}\right\Vert_2 ^{m}$, and
Condition (\ref{eq:optimality}) is satisfied. The same calculation
shows that Condition $\boldsymbol{A_2'}$ is satisfied: for all
$\boldsymbol{\theta}\in\Theta_{j_{n}}$,
$\widetilde{K}\left(p_{0j_{n}}^{(n)},p_{\boldsymbol{\theta}}^{(n)}\right)=\frac{n}{2}\left\Vert
\boldsymbol\theta_{0j_{n}}-\boldsymbol{\theta}\right\Vert_{2,j_{n}}
^{2}$ and
$\widetilde{V}_{m,0}\left(p_{0j_{n}}^{(n)},p_{\boldsymbol{\theta}}^{(n)}\right)\lesssim
n^{m/2}\left\Vert
\boldsymbol\theta_{0j_{n}}-\boldsymbol{\theta}\right\Vert_{2,j_{n}}
^{m}$.

Conditions $\boldsymbol{A_2}$ - $\boldsymbol{A_4}$ and Condition
(\ref{eq:optimality}) hold, if moreover $\boldsymbol{A_4}$ is satisfied, then by Proposition \ref{prop:optimal},
the procedure is adaptive, which is expressed in the following
proposition.
\begin{prop}
Under the prior $\Pi$ defined in Equations (\ref{eq:priorexample}),
the global $l^2$ rate of posterior contraction is optimal adaptive
for the Gaussian white noise model, \textit{i.e.} for $M$ large
enough and $\beta_2\geq\beta_1>1/2$\[
\sup_{\beta_1\leq\beta\leq\beta_2}\sup_{\boldsymbol{\theta}_0\in
\Theta_\beta(L_0)}
\E_{0}^{(n)}\Pi\left(\boldsymbol{\theta}:\,\left\Vert
\boldsymbol{\theta}-\boldsymbol{\theta}_{0}\right\Vert_2 ^{2}\geq
M\frac{\log n}{L(n)}\epsilon_{n}^{2}(\beta)\vert
X^{n}\right)\underset{n\rightarrow\infty}{\longrightarrow}0,\] with
$\epsilon_{n}(\beta) = \epsilon_{0}\left(\frac{\log
n}{n}\right)^{\frac{\beta}{2\beta+1}}$.
\end{prop}
The distance here is not bounded, so Corollary \ref{cor:point} does
not hold. For deriving a risk rate, we need a more subtle result
than Theorem \ref{thm:apo} that we can obtain when considering sets
$\mathcal{S}_{n,j}(M) = \left\{ \boldsymbol{\theta}:M\frac{\log
n}{L(n)} (j+1)\epsilon_{n}^{2} \geq \left\Vert
\boldsymbol{\theta}-\boldsymbol{\theta}_{0}\right\Vert_2 ^{2}\geq M
\frac{\log n}{L(n)} j\epsilon_{n}^{2}\right\}$, $j=1,2,\ldots$
instead of $\mathcal{S}_{n}(M) = \left\{
\boldsymbol{\theta}:\left\Vert
\boldsymbol{\theta}-\boldsymbol{\theta}_{0}\right\Vert_2 ^{2}\geq M
\frac{\log n}{L(n)} \epsilon_{n}^{2}\right\} $. Then the bound of
the expected posterior mass of $\mathcal{S}_{n,j}(M)$ becomes
\begin{equation}
\E_{0}^{(n)}\Pi(\mathcal{S}_{n,j}(M)|X^n)\leq
c_{7}\left(nj\epsilon_{n}^{2}\right)^{-m/2}\label{eq:j}\end{equation}
for a fixed constant $c_{7}$. Hence we obtain the following rate of
convergence in risk.
\begin{prop}
\label{cor:point-1}Under Condition (\ref{eq:optimality}) with
$m\geq5$, the expected posterior risk given
$\boldsymbol{\theta}_{0}$ and $\Pi$ converges at least at the same
rate $\epsilon_{n}$
\[
\mathcal{R}_{n}^{L^2}(\boldsymbol{\theta}_{0})=\E_{0}^{(n)}\Pi\left[\left\Vert
\boldsymbol{\theta}-\boldsymbol{\theta}_{0}\right\Vert_2 ^{2}\vert
X^{n}\right]=\go\left(\epsilon_{n}^{2}\right),\] for any
$\boldsymbol{\theta}_{0}$. So the procedure is risk adaptive as well
(up to a $\log(n)$ term).\end{prop}
\begin{proof}
We have\begin{eqnarray*}
\mathcal{R}_{n}^{L^2}(\boldsymbol{\theta}_{0}) & \leq &
\E_{0}^{(n)}\Pi\left[\left( \ind(\boldsymbol{\theta}\notin
\mathcal{S}_{n}(M)) + \sum_{j\geq 1}\ind(\boldsymbol{\theta}\in
\mathcal{S}_{n,j}(M))
\right)\left\Vert \boldsymbol{\theta}-\boldsymbol{\theta}_{0}\right\Vert_2 ^{2}\vert X^{n}\right]\\
 & \leq & M \frac{\log n}{L(n)} \epsilon_{n}^{2}\left(1+\sum_{j=1}^{\infty}(j+1)\E_{0}^{(n)}\Pi(\Sc_{n,j}(M)\vert X^{n})\right).\end{eqnarray*}

Due to (\ref{eq:j}), the last sum in $j$ converges as soon as
$m\geq5$. This is possible in the white noise setting because the
conditions are satisfied whatever $m$. So
$\mathcal{R}_{n}^{L^2}(\boldsymbol{\theta}_{0})=\go\left(\epsilon_{n}^{2}\right)$.
\end{proof}

We have shown that conditional to the existence of a sieve prior for
the white noise model satisfying $\boldsymbol{A_5}$ (\emph{cf.}
Section \ref{sec:examples}), the procedure has minimax rates (up to
a $\log(n)$ term) both in contraction and in risk. We now study the
asymptotic behaviour of the posterior under the local loss function.

\subsection{Lower bound under pointwise loss}

The previous section derives rates of convergence under the global
loss. Here, under the pointwise loss, we show that the risk
deteriorates as a power $n$ factor compared to the benchmark minimax
pointwise risk $n^{-(2\beta-1)/2\beta}$ (note the difference
with the global minimax rate $n^{-2\beta/(2\beta+1)}$, both
given for risks on squares). We use the sieve prior defined as a
conditional Gaussian prior in Equation (\ref{eq:priorexample}).
Denote by $\hat{\boldsymbol{\theta}}_n$ the Bayes estimate of
$\boldsymbol{\theta}$ (the posterior mean). Then the
following proposition gives a lower bound on the risk (pointwise square
error) under a pointwise loss:
\begin{prop}
\label{pro:Under-pointwise-loss,}If the point $t$ is such that
$a_j=\psi_j(t)=1$ for all $j$ ($t=0$), then for all $\beta\geq q$,
for all $L_0>0$, a lower bound on the risk rate under pointwise loss
is given by\[
\sup_{\boldsymbol{\theta}_{0}\in\Theta_{\beta}(L_{0})}R_{n}^{\emph{loc}}(\boldsymbol{\theta_0},t)\gtrsim n^{-\frac{2\beta-1}{2\beta+1}}/\log^{2}n.\]
\end{prop}
\begin{proof}
See the Appendix.
\end{proof}
\citet{Cai:2007p90} show that a global optimal estimator cannot be
pointwise optimal. The sieve prior leads to an (almost up to a $\log
n$ term) optimal global risk and Proposition
\ref{pro:Under-pointwise-loss,} shows that the pointwise risk
associated to the posterior mean $\hat{\boldsymbol{\theta}}_n$ is
suboptimal with  a power of $n$ penalty, whose exponent is
$$\frac{2\beta -1}{2\beta } - \frac{2\beta -1}{2\beta
+1}=\frac{2\beta -1}{2\beta (2\beta+1)}.$$ The maximal penalty is for
$\beta=(1+\sqrt{2})/2$, and it vanishes as $\beta$ tends to
1/2 and $+\infty$ (see the Figure \ref{fig:exp}).
\citet{Abramovich:2007p6014} also derive such a power $n$ penalty on
the maximum local risk of a globally optimal Bayesian estimate, as
well as on the reverse case (maximum global risk of a locally
optimal Bayesian estimate).
\begin{rem}
    This result is not anecdotal and illustrates the fact that the
    Bayesian approach is well suited for loss functions that are
    related to the Kullback-Leibler divergence (\textit{i.e.} often
    the $l^2$ loss). The pointwise loss does not satisfy this since
    it corresponds to an unsmooth linear functional of
    $\boldsymbol{\theta}$.  
    This possible suboptimality of the posterior distribution of
    some unsmooth functional of the parameter has already been
    noticed in various other cases, see for instance
    \cite{Rivoirard:2011} or \cite{kruijer_spectral}. The question of the
    existence of a fully Bayesian adaptive procedure to estimate $\boldsymbol{f}_0(t)=\sum_{j=1}^\infty a_j \theta_{0j}$ remains an
    open question.
\end{rem}

\begin{figure}
\begin{center}
\includegraphics[width=8cm,height=6cm]{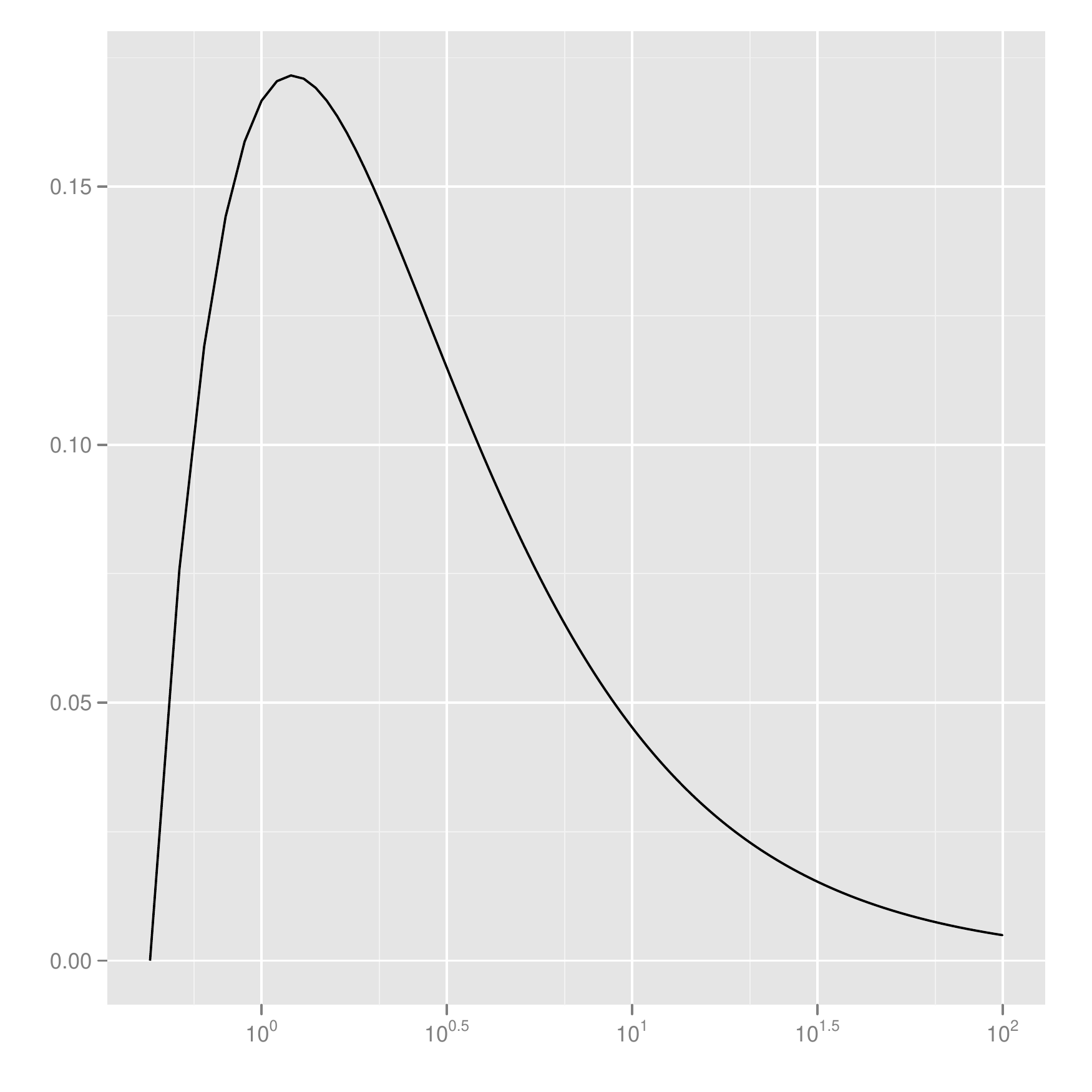}\\
  \caption{Variation of the exponent of the penalty in a log scale for $\beta$ between 1/2 and 100; it is maximum for $\beta=(1+\sqrt{2})/2$}\label{fig:exp}
\end{center}
\end{figure}

\section*{Acknowledgements}

We would like to thank the referees for their valuable comments which have helped to improve the manuscript.

\bibliographystyle{apalike}
\bibliography{biblio_short}

Corresponding author: Julyan Arbel, E37, Laboratoire de Statistique, CREST, 92245 Malakoff, France. \\
E-mail: \textit{julyan.arbel@m4x.org}

\appendix

\section{Appendix}

\subsection{Three
technical lemmas}

Set $\Sc_{n}(M)=\{ \boldsymbol{\theta}:\,
d_{n}^{2}(\boldsymbol{\theta},\boldsymbol{\theta}_{0})\geq
M\frac{\log n}{L(n)}\epsilon_{n}^{2}\} $ and recall that
$\Theta_{k_n}(Q) = \{ \boldsymbol{\theta} \in \Theta_{k_n} :\,\Vert
\boldsymbol{\theta}\Vert_{2,k_n} \leq n^Q\}$, $Q>0$. We begin with three
technical lemmas.

\begin{lem}
\label{lem:tests}If Conditions $\boldsymbol{A_3}$ and
$\boldsymbol{A_4}$ hold, then there exists a test
$\boldsymbol{\phi}_{n}$ 
such that for $M$ large enough, there exists a constant $c_2$ such that
\[
\E_{0}^{(n)}(\boldsymbol{\phi}_{n})\leq e^{-c_{2}M\frac{\log
n}{L(n)}n\epsilon_{n}^{2}}\quad
and\quad\E_{\boldsymbol{\theta}}^{(n)}\left(1-\boldsymbol{\phi}_{n}\right)\leq
e^{-c_{2}M\frac{\log n}{L(n)}n\epsilon_{n}^{2}},\] for all
$\boldsymbol{\theta}\in\Sc_{n}(M)\cap\Theta_{k_n}(Q)$.\end{lem}
\begin{proof}
Set $r_{n}=\left(\sqrt{M\frac{\log
n}{L(n)}}\frac{\zeta\epsilon_{n}}{D_0k_{n}^{D_{1}}}\right)^{1/D_{2}}$.
The set $\Sc_{n}(M)\cap\Theta_{k_n}(Q)$ is compact relative to the
$l^2$ norm. Let a covering of this set by $l^2$ balls of radius
$r_{n}$ and centre $\boldsymbol{\theta}^{(i)}$. Its number of
elements is
$\eta_n\lesssim(Cn^Q/r_{n})^{k_{n}}\lesssim\exp(Ck_{n}\log
n)\lesssim\exp(C\frac{\log n}{L(n)}n\epsilon_{n}^{2})$ due to
relation (\ref{kn}).

For each centre $\boldsymbol{\theta}^{(i)}\in\Sc_{n}(M)\cap\Theta_{k_n}(Q)$, there exists a test
$\phi_{n}(\boldsymbol{\theta}^{(i)})$ satisfying Condition
$\boldsymbol{A_4}$. We define the test
$\boldsymbol{\phi}_{n}=\max_{i}\phi_{n}(\boldsymbol{\theta}^{(i)})$
which satisfies
\[
\E_{0}^{(n)}(\boldsymbol{\phi}_{n})\leq
\eta_ne^{-c_{1}M\frac{\log
n}{L(n)}n\epsilon_{n}^{2}}\leq e^{C\frac{\log
n}{L(n)}n\epsilon_{n}^{2}-c_{1}M\frac{\log
n}{L(n)}n\epsilon_{n}^{2}}\leq
e^{-c_{2}M\frac{\log
n}{L(n)}n\epsilon_{n}^{2}},\] for $M$ large enough and a constant $c_2$.

Here, Condition $\boldsymbol{A_3}$ allows to switch from the
coverage in term of the $l^2$ distance to a covering expressed in
term of $d_{n}$: each
$\boldsymbol{\theta}\in\Sc_{n}(M)\cap\Theta_{k_n}(Q)$ which lies in
a $l^2$ ball of centre $\boldsymbol{\theta}^{(i)}$ and of radius
$r_{n}$ in the covering of size $\eta_n$ also lies in a $d_{n}$ ball
of adequate radius
\[
d_{n}(\boldsymbol{\theta},\boldsymbol{\theta}^{(i)})\leq
D_{0}k_{n}^{D_{1}}\Vert
\boldsymbol{\theta}-\boldsymbol{\theta}^{(i)}\Vert_2 ^{D_{2}}\leq
D_{0}k_{n}^{D_{1}}r_{n}^{D_{2}}=\zeta\epsilon_{n}\sqrt{M\frac{\log
n}{L(n)}}.\]
 Then there exists a constant $c_2$ (the minimum with the previous one)
\[
\sup_{\boldsymbol{\theta}\in\Sc_{n}(M)\cap\Theta_{k_n}(Q)}\E_{\boldsymbol{\theta}}^{(n)}\left(1-\boldsymbol{\phi}_{n}\right)\leq
e^{-c_{2}M\frac{\log
n}{L(n)}n\epsilon^{2}},\] hence the result follows.\end{proof}

\begin{lem}
\label{lem:prior_mass}Under Condition $\boldsymbol{A_5}$, for any constant $c_6>0$, there
exist positive constants $Q$,  $C$ and $M_0$ such that
\begin{equation}
\Pi(\Theta_{k_n}^{c}(Q))\leq
Ce^{-c_{6}n\epsilon_{n}^{2}},\label{eq:reslem2}\end{equation} where $M_0$ is introduced in the definition (\ref{kn}) of $k_n$, and 
$\Theta_{k_n}^{c}(Q)$, the complementary of $\Theta_{k_n}(Q)$, is
taken in $\Theta$.\end{lem}

\begin{proof}
$\Theta_{k_n}^{c}(Q)$ is written by $\Theta_{k_n}^c(Q)=\{
\boldsymbol{\theta}\in\Theta:\,\Vert
\boldsymbol{\theta}\Vert_{2,k_n} >  n^Q  \text{ or } \exists j>k_n
\text{ s.t. } \theta_j\neq 0\}$, so its prior mass is less than
$\pi(k> k_{n}) + \sum_{k\leq
k_n}\pi_k\Pi_{k}(\boldsymbol{\theta}\in\Theta_{k}:\,\Vert
\boldsymbol{\theta}\Vert_{2,k}
> n^Q )$, where the last sum is less than $\Pi_{k_n}(\boldsymbol{\theta}\in\Theta_{k_n}:\,\Vert
\boldsymbol{\theta}\Vert_{2,k_n}
> n^Q )$ because its terms are increasing.

The prior mass of sieves that exceed $k_n$ is controlled by Equation
(\ref{eq:prior k}). We have
\begin{equation*} \pi\left(k\geq k_{n}\right)\leq \sum_{j\geq k_n} e^{-bjL(j)} \leq \sum_{j\geq k_n} e^{-bjL(k_n)} \leq Ce^{-bk_nL(k_n)}.\end{equation*}

Since $L$ is a slow varying function, we have $k_nL(k_n)\gtrsim
j_n\log(n)\gtrsim n\epsilon_n^2$. Hence $\pi\left(k\geq
k_{n}\right)\leq Ce^{-c_{6}n\epsilon_{n}^{2}}$ for a constant $c_6$
as large as needed since it is determined by constant $M_0$ in Equation (\ref{kn}).

Then by the second part of Condition (\ref{eq:priorg}), $
\Pi_{k_{n}}\left(\boldsymbol{\theta}\in\Theta_{k_n}:\,\Vert
\boldsymbol{\theta}\right\Vert_{2,k_n}
> n^Q )$ is less than
\begin{eqnarray}
 &  & \int_{\Vert \boldsymbol{\theta}\Vert_{2,k_n} > n^Q }\prod_{j=1}^{k_{n}}\, g(\theta_{j}/\tau_{j})/\tau_{j}d \theta_{j},\nonumber\\
&\leq & (G_{3}n^{H_2})^{k_{n}}\int_{\Vert
\boldsymbol{\theta}\Vert_{2,k_n} > n^Q
}\exp(-G_{4}\sum_{j=1}^{k_{n}}|\theta_{j}|^{\alpha
}/\tau_{j}^{\alpha })\, d \theta_{i},\label{eq:upperbound}\end{eqnarray} by
using the lower bound on the $\tau_j$'s of Equation
(\ref{eq:mintau}).

 If $\alpha \geq 2$, then
applying H\"{o}lder inequality, one obtains
$$n^{2Q}\leq \Vert \boldsymbol{\theta}\Vert_{2,k_n}^2\leq \Vert \boldsymbol{\theta}\Vert_{\alpha,k_n}^2 k_n^{1-2/\alpha },$$
which leads to
$$\Vert \boldsymbol{\theta}\Vert_{\alpha,k_n}^\alpha\geq
k_{n}^{1-\alpha /2}n^{Q\alpha }.$$

If $\alpha < 2$, then a classical result states that the $l^\alpha$
norm $\Vert\,.\,\Vert_\alpha$  is larger than the $l^2$ norm
$\Vert\,.\,\Vert_2$, \textit{i.e.}
$$\Vert \boldsymbol{\theta}\Vert_{\alpha,k_n}^\alpha\geq \Vert \boldsymbol{\theta}\Vert_{2,k_n}^\alpha \geq
n^{Q\alpha}.$$

Eventually the upper bound $\tau_0$ on the $\tau_j$'s of Equation
(\ref{eq:maxtau}) provides
$$ \sum_{j=1}^{k_{n}}\left|\theta_{j}\right|^{\alpha}/\tau_j^{\alpha}\geq \tau_0^{-\alpha}n^{Q\alpha}\min(k_{n}^{1-\alpha/2},1).$$

The integral in the right-hand side of (\ref{eq:upperbound}) is bounded by

$$\exp(-\frac{G_4}{2}\Vert \boldsymbol{\theta}\Vert_{2,k_n}^\alpha / \tau_0^{\alpha})\int_{\Theta_{k_n}} \exp(-\frac{G_4}{2}\sum_{j=1}^{k_{n}}|\theta_{j}|^{\alpha
}/\tau_{j}^{\alpha })\, d \theta_{i}.$$

The last integral is bounded by $C^{k_n}$, so
\begin{equation*}
\Pi_{k_{n}}\left(\boldsymbol{\theta}\in\Theta_{k_n}:\,\left\Vert \boldsymbol{\theta}\right\Vert_{2,k_n}
> n^Q \right)\leq C^{k_{n}\log n}\exp(-\frac{G_4}{2}\tau_{0}^{-\alpha}n^{Q\alpha}\min(k_{n}^{1-\alpha/2},1)).
\end{equation*}
The right-hand side of the last inequality can be made smaller than $Ce^{-c_{6}n\epsilon_{n}^{2}}$ for any constant $C$ and $c_6$ provided that 
$Q$ is chosen large enough. This entails result (\ref{eq:reslem2}).

{In the truncated case (\ref{eq:gntrunc}), we note that if $\sum_{j=1}^{k_n} \vert \theta_j \vert \leq r_1$, then $\sum_{j=1}^{k_n}  \theta_j^2 \leq r_1^2$, so that for $n$ large enough, $\Pi(\Theta_{k_n}^{c}(Q))=\pi(k\geq k_n)$, and the rest of the proof is similar.}
\end{proof}

\begin{lem}
\label{lem:minoration_bn}Under Conditions $\boldsymbol{A_1}$,
$\boldsymbol{A_2}$ and $\boldsymbol{A_5}$, there exists $c_{4}>0$
such that \[ \Pi(\mathcal{B}_{n}(m))\geq
e^{-c_{4}n\epsilon_{n}^{2}}.\]
\end{lem}
\begin{proof}
Let $\boldsymbol{\theta}\in\Ac_{n}(H_1)$. For $n$ large enough,
Conditions $\boldsymbol{A_1}$ and $\boldsymbol{A_2}$ imply that\[
K(p_{0}^{(n)},p_{\boldsymbol{\theta}}^{(n)})\leq
K(p_{0}^{(n)},p_{0j_{n}}^{(n)})+\widetilde{K}(p_{0j_{n}}^{(n)},p_{\boldsymbol{\theta}}^{(n)})\leq
2n\epsilon_{n}^{2},\] and\begin{eqnarray*}
V_{m,0}(p_{0}^{(n)},p_{\boldsymbol{\theta}}^{(n)}) & = & \int p_{0}^{(n)}\left|\log({p_{0}^{(n)}}/{p_{0j_{n}}^{(n)}})-K(p_{0}^{(n)},p_{0j_{n}}^{(n)})+\right.\\
 &  & \left.\log({p_{0j_{n}}^{(n)}}/{p_{\boldsymbol{\theta}}^{(n)}})-\int p_{0}^{(n)}\log({p_{0j_{n}}^{(n)}}/{p_{\boldsymbol{\theta}}^{(n)}})d\mu\right|^{m}d\mu\\
 & \leq & 2^{m}(V_{m,0}(p_{0}^{(n)},p_{0j_{n}}^{(n)})+\widetilde{V}_{m,0}(p_{0j_{n}}^{(n)},p_{\boldsymbol{\theta}}^{(n)}))\leq 2^{m+1}\left(n\epsilon_{n}^{2}\right)^\frac{m}{2},\end{eqnarray*}
which yields $\Ac_{n}(H_1)\subset\mathcal{B}_{n}(m)$ so that a lower
bound for $\Pi(\mathcal{B}_{n}(m))$ is given by $\Pi(\Ac_{n}(H_1))$.
Note that for $H_0>H_1$, then
\begin{equation}
\label{eq:AA}
\Ac_{n}(H_0)\subset\Ac_{n}(H_1)\subset\mathcal{B}_{n}(m).
\end{equation}
We have 
\begin{equation*}
\Pi(\Ac_{n}(H_1))=\sum_{k=1}^\infty\pi(k)\Pi_{k}(\Ac_{n}(H_1))\geq \pi(j_n)\Pi_{j_n}(\Ac_{n}(H_1)).
\end{equation*}
By the first part of Condition (\ref{eq:prior k}) we have 
\begin{equation}
\label{eq:kna}
\pi(j_n)\geq e^{-j_nL(j_n)}\geq e^{-\frac{c_4}{2}n\epsilon_n^2},
\end{equation}
 for $c_4$ large enough. 
Now by the first part of Condition (\ref{eq:priorg}) and by
Condition (\ref{eq:maxtau})
\begin{eqnarray}
\Pi_{j_n}(\Ac_{n}(H_1)) & = & \int_{\Vert \boldsymbol{\theta}-\boldsymbol{\theta}_{0j_{n}}\Vert_{2,j_n} \leq n^{-H_1}}\prod_{j=1}^{j_{n}}\,g({\theta_{j}}/{\tau_{j}})/\tau_{j}d\theta_{j}\label{eq:int_A_n}\\
 & \geq & \left(G_{1}/\tau_{0}\right)^{j_{n}}
 \int_{\left\Vert \boldsymbol{\theta}-\boldsymbol{\theta}_{0j_{n}}\right\Vert_{2,j_n} \leq n^{-H_1}}\exp(-G_{2}\sum_{j=1}^{j_{n}}|\theta_{j}|^{\alpha}/\tau_{j}^{\alpha})d\theta_{j}.\nonumber
\end{eqnarray}
We can bound above $\tau_{j}^{-\alpha }$ by $n^{\alpha H_2}$ by
Equation (\ref{eq:mintau}) as $j\leq j_n\leq k_n$. We write
$\left|\theta_{j}\right|^{\alpha }\leq2^{\alpha
}\left(\left|\theta_{0j}\right|^{\alpha
}+\left|\theta_{j}-\theta_{0j}\right|^{\alpha }\right)$. First,
Equation (\ref{eq:somme theta tau}) gives \[
\sum_{j=1}^{j_{n}}\left|\theta_{0j}\right|^{\alpha
}/\tau_{j}^{\alpha }\leq Cj_{n}\log n.\] Then, if $\alpha \geq2$ \[
\sum_{j=1}^{j_{n}}\left|\theta_{j}-\theta_{0j}\right|^{\alpha
}\leq\left\Vert
\boldsymbol{\theta}-\boldsymbol\theta_{0j_n}\right\Vert_{2,j_n}
^{\alpha}\leq n^{-\alpha H_1},\] and if $\alpha <2$ then
H\"{o}lder's inequality provides\[
\sum_{j=1}^{j_{n}}\left|\theta_{j}-\theta_{0j}\right|^{\alpha
}\leq\left\Vert
\boldsymbol{\theta}-\boldsymbol\theta_{0j_n}\right\Vert_{2,j_n}
^{\alpha }j_{n}^{1-\alpha /2}\leq n^{-\alpha H_1}j_{n}^{1-\alpha
/2}.\] In both cases we have\[
\sum_{j=1}^{j_{n}}\left|\theta_{j}\right|^{\alpha }/\tau_{j}^{\alpha
}\leq 2^\alpha (Cj_{n}\log n+n^{\alpha (H_2-H_1)}j_{n}^{1-\alpha
/2}),\] so choosing $H_2 \leq H_1$ ensures to bound the latter by
$j_n\log n$. Last, the integral of the ball in dimension $j_{n}$,
centered around $\boldsymbol{\theta}_{0j_{n}}$, and of radius
$n^{-H_1}$, is at least equal to $e^{-Cj_{n}\log n}$, for some given
positive constant $C$.

Noting that $j_{n}= \lfloor j_{0}n\epsilon_{n}^{2}/\log(n)\rfloor$
and choosing $H_1$ large enough, which is possible by Equation
(\ref{eq:AA}), ensures the existence of $c_{4}>0$ such that
$\Pi_{j_n}(\Ac_{n}(H_1))\geq e^{-\frac{c_{4}}{2} n\epsilon_{n}^{2}}$. Combining this with  (\ref{eq:kna}) allows to conclude.

{In the truncated case (\ref{eq:gntrunc}), we can first choose $r_1$ larger than $2\sum_{j=1}^{j_n} \vert \theta_{0j} \vert$. If $\boldsymbol{\theta}\in\Ac_{n}(H_1)$, then $\sum_{j=1}^{j_n} \vert \theta_{j} \vert\leq \sum_{j=1}^{j_n} (\vert \theta_{j}-\theta_{0j} \vert+\vert \theta_{0j} \vert)\leq \sqrt{j_n}n^{-H_1}+r_1/2\leq r_1$ for $n$ and $H_1$ large enough. 
So the expression of integral (\ref{eq:int_A_n}) is still valid.}
\end{proof}

\subsection{Theorem \ref{thm:apo}}

\begin{proof}
(of \textbf{Theorem \ref{thm:apo}})

Express the quantity of interest $\Pi\left(\Sc_{n}(M)\vert
X^{n}\right)$ in terms of $N_n$, $\widetilde{N_{n}}$ and $D_{n}$
defined as follows
\begin{equation*}
\frac{\int_{\Sc_{n}(M)\cap\Theta_{k_n}(Q)}p_{\boldsymbol{\theta}}^{(n)}/p_{\boldsymbol{\theta}_{0}}^{(n)}d\Pi(\boldsymbol{\theta})+\int_{\Sc_{n}(M)\cap\Theta_{k_n}^{c}(Q)}p_{\boldsymbol{\theta}}^{(n)}/p_{\boldsymbol{\theta}_{0}}^{(n)}d\Pi(\boldsymbol{\theta})}{\int_{\Theta}p_{\boldsymbol{\theta}}^{(n)}/p_{\boldsymbol{\theta}_{0}}^{(n)}d\Pi(\boldsymbol{\theta})}\\
  :=  \dfrac{N_{n}+\widetilde{N_{n}}}{D_{n}}.\end{equation*}
Denote
$\rho_{n}(c_{3})=\exp(-(c_{3}+1)n\epsilon_{n}^{2})\Pi(\mathcal{B}_{n}(m))$
for $c_{3}>0$. Introduce $\boldsymbol{\phi}_n$ the test statistic of
Lemma \ref{lem:tests}, and take the expectation of the posterior mass of
$\Sc_{n}(M)$ as follows
\begin{align}
 & \E_{0}^{(n)}\left(\frac{N_{n}+\widetilde{N_{n}}}{D_{n}}\left(\boldsymbol{\phi}_{n}+1-\boldsymbol{\phi}_{n}\right)\left(\ind(D_{n}\leq\rho_{n}(c_{3}))+\ind(D_{n}>\rho_{n}(c_{3}))\right)\right) \nonumber \\
 & \leq  \E_{0}^{(n)}\left(\boldsymbol{\phi}_{n}\right)+\E_{0}^{(n)}\left(\frac{N_{n}+\widetilde{N_{n}}}{D_{n}}\left(1-\boldsymbol{\phi}_{n}\right)\left(\ind(D_{n}\leq\rho_{n}(c_{3}))+
 \ind(D_{n}>\rho_{n}(c_{3}))\right)\right) \nonumber\\
 & \leq  \E_{0}^{(n)}\left(\boldsymbol{\phi}_{n}\right)+ \p_{0}^{(n)}\left(D_{n}\leq\rho_{n}(c_{3})\right)+\frac{\E_{0}^{(n)}\left(N_{n}\left(1-\boldsymbol{\phi}_{n}\right)\right)+
 \E_{0}^{(n)}(\widetilde{N_{n}})}{\rho_{n}(c_{3})}.\label{eq:post mass Sn}
 \end{align}
Lemma 10 in \citet{Ghosal:2007p267} gives
$\p_{0}^{(n)}\left(D_{n}\leq\rho_{n}(c_{3})\right)\lesssim\left(n\epsilon_{n}^{2}\right)^{-m/2}$
for every $c_{3}>0$.

Fubini's theorem entails that
$\E_{0}^{(n)}(N_{n}(1-\boldsymbol{\phi}_{n}))\leq\sup_{\Sc_{n}(M)\cap\Theta_{k_n}(Q)}\E_{\boldsymbol{\theta}}^{(n)}(1-\boldsymbol{\phi}_{n})$.
Along with $\E_{0}^{(n)}(\boldsymbol{\phi}_{n})$, it is upper
bounded in Lemma \ref{lem:tests} by $e^{-c_{2}M\frac{\log
n}{L(n)}n\epsilon_{n}^{2}}$.

Lemma \ref{lem:prior_mass} implies that
$\E_{0}^{(n)}(\widetilde{N_{n}})\leq\Pi(\Theta_{k_n}^{c}(Q))\leq
e^{-c_{6}n\epsilon_{n}^{2}}$ and Lemma \ref{lem:minoration_bn}
yields $\Pi_{n}(\mathcal{B}_{n}(m))\geq
e^{-c_{4}n\epsilon_{n}^{2}}$. 
Constants $c_3$ and $c_4$ are fixed, so we can choose $M$, $M_0$ and $Q$ large enough for $c_6$ to be sufficiently large (see proof of Lemma \ref{lem:prior_mass}), such
that $\min(M\frac{\log n}{L(n)}c_2,c_6)>c_3+c_4+1$. It implies that
the third term in Equation (\ref{eq:post mass Sn}) is bounded above
by $e^{-c_{5}n\epsilon_{n}^{2}}$ for some positive $c_{5}$. Finally,
\[
\E_{0}^{(n)}\Pi\left(\Sc_{n}(M)\vert
X^{n}\right)=\mathcal{O}\left(\left(n\epsilon_{n}^{2}\right)^{-m/2}\right)\underset{n\to\infty}{\longrightarrow}0,\]
since
$n\epsilon_{n}^{2}\underset{n\rightarrow\infty}{\longrightarrow}\infty$.
\end{proof}

\subsection{Proposition \ref{pro:Under-pointwise-loss,}}

The proof of the lower bound in the local risk case uses the next
lemma, whose proof follows from Cauchy-Schwarz' inequality.

\begin{lem}
\label{lem:lowerbound} If $\E(B_n^2)=o(\E(A_n^2))$, then
$\E((A_n+B_n)^2)=\E(A_n^2)(1+o(1))$.
\end{lem}

\begin{proof}
(of \textbf{Proposition \ref{pro:Under-pointwise-loss,}})

The coordinates of $\hat{\boldsymbol{\theta}}_n$ are
$\hat{\theta}_{nj} = \Pi\left({\theta}_j\vert
X^{n}\right) = \sum_{k=1}^{\infty}\pi(k\vert
X^{n})\widetilde{\theta}_{nj}(k)$, with
$\widetilde{\theta}_{nj}(k) = {\tau_{j}^{2}}/(\tau_{j}^{2}+\nicefrac{1}{n})X_{j}^{n}$
if $k\geq j$, and $\widetilde{\theta}_{nj}(k)=0$ otherwise
\citep[see][]{Zhao:2000p98}.

Denote $u_{j}(X^{n})=\sum_{k\geq j}\pi(k\vert X^{n})=\pi(k\geq
j\vert X^{n})$, so that
$\hat{\theta}_{nj}=u_{j}(X^{n}){\tau_{j}^{2}}/(\tau_{j}^{2}+\nicefrac{1}{n})X_{j}^{n}$.
Denote $K_{n}=n^{1/(2\beta+1)}$ and
$J_{n}=n^{1/2\beta}$. {Most of
the posterior mass on $k$ is concentrated before
$K_{n}$, in the sense that there exists a
constant $c$ such that
\begin{equation}
\E_{0}^{(n)}\left(u_{K_{n}}(X^{n})\right)\lesssim\exp\left(-cK_{n}\right).\label{eq:E_K_n}\end{equation}
This follows from the exponential inequality
\begin{equation*}
P_{\boldsymbol{\theta}_0}^{(n)}[u_{K_{n}}(X^{n})>\exp(-cK_{n})]\lesssim\exp(-cK_{n}),
\end{equation*}
which is obtained by classic arguments in line with Theorem \ref{thm:apo}: writing the posterior quantity $u_{K_{n}}(X^{n})$ as a ratio $N_n/D_n$, and then using Fubini's theorem, Chebyshev's inequality and an upper bound on $\pi(k>K_n)$.}

Due to Relation (\ref{eq:wnm}), we split in three the sum in the
risk $$R_{n}^{\text{loc}}(\boldsymbol{\theta}_{0},t) = \E_{0}^{(n)}
\left(\sum_{i=1}^{\infty}a_{i}[(1-u_{i}(X^{n})\tfrac{\tau_{i}^{2}}{\tau_{i}^{2}
+
\nicefrac{1}{n}})\theta_{0i}-u_{i}(X^{n})\tfrac{\tau_{i}^{2}}{\tau_{i}^{2}
+ \nicefrac{1}{n}}\frac{\xi_{i}}{\sqrt{n}}]\right)^{2}$$ by
centring the stochastic term $X_{i}^{n}$ and writing
$1-u_{i}(X^{n})\tfrac{\tau_{i}^{2}}{\tau_{i}^{2}+\nicefrac{1}{n}}=\frac{1}{n}\tfrac{1}{\tau_{i}^{2}+\nicefrac{1}{n}}+\tfrac{\tau_{i}^{2}}{\tau_{i}^{2}+\nicefrac{1}{n}}(1-u_{i}(X^{n}))$.
The idea of the proof is to show that there is a leading term in the
sum, and to apply Lemma \ref{lem:lowerbound}.

Let $R_1=\left(\sum_{i=1}^{\infty}a_{i}\tfrac{1}{n\tau_{i}^{2}+1}\theta_{0i}\right)^{2}$,
$R_2 = \E_{0}^{(n)}\left(\sum_{i=1}^{\infty}a_{i}\tfrac{\tau_{i}^{2}}{\tau_{i}^{2}+\nicefrac{1}{n}}(1-u_{i}(X^{n}))\theta_{0i}\right)^{2}$
 and $R_3 =
 \E_{0}^{(n)}\left(\sum_{i=1}^{\infty}a_{i}\tfrac{\tau_{i}^{2}}{\tau_{i}^{2}+\nicefrac{1}{n}}u_{i}(X^{n})\frac{\xi_{i}}{\sqrt{n}}\right)^{2}$.
By using Cauchy-Schwarz' inequality
\begin{eqnarray*}
R_{1}=\left(\sum_{i=1}^{\infty}a_{i}\tfrac{1}{n\tau_{i}^{2}+1}\theta_{0i}\right)^{2} & = & \left(\sum_{i=1}^{\infty}a_{i}\tfrac{i^{-\beta}}{n\tau_{i}^{2}+1}\theta_{0i}i^{\beta}\right)^{2}\\
 & \lesssim & L_{0}\sum_{i=1}^{\infty}\tfrac{i^{-2\beta}}{(ni^{-2q}+1)^{2}},\end{eqnarray*}
because the $a_{i}$'s are bounded. If $2\beta-4q>1$, then we can write
\[
R_{1}\lesssim\frac{1}{n^{2}}\sum_{i=1}^{\infty}i^{-2\beta+4q}\lesssim\frac{1}{n^{2}},\]
and if $2\beta-4q\leq1$, then comparing to an integral
provides\begin{eqnarray*}
R_{1} & \lesssim & \int_{1}^{\infty}\tfrac{x^{-2\beta}}{(nx^{-2q}+1)^{2}}dx\\
 & \lesssim & \left(n^{1/2q}\right)^{1-2\beta}\int_{n^{-1/2q}}^{\infty}\tfrac{y^{-2\beta}}{(y^{-2q}+1)^{2}}dy\\
 & \lesssim & n^{-\frac{2\beta-1}{2q}}\lesssim n^{-\frac{2\beta-1}{2\beta}},\end{eqnarray*}
where the last inequality holds because $q$ is chosen such that
$q\leq\beta$. Then
$R_{1}=\go(n^{-(2\beta-1)/2\beta})$.

For $k=2,3$, denote $R_{k}(b_{n},c_{n})$ the partial sum of $R_{k}$
from $j=b_{n}$ to $c_{n}$. Then $R_{2}(1,J_{n})$ is the larger term
in the decomposition, and is treated at the end of the section. The
upper part $R_{2}(J_{n},\infty)$ is easily bounded by
\[
R_{2}(J_{n},\infty)\lesssim\left(\sum_{i=J_{n}}^{\infty}\left|\theta_{0i}\right|i^{\beta}i^{-\beta}\right)^{2}\lesssim\,
J_{n}^{-2\beta+1}=\go\left(n^{-\frac{2\beta-1}{2\beta}}\right).\] We
split $R_{3}(1,J_{n})$ in two parts $R_{3,1}(1,J_{n})$ and
$R_{3,2}(1,J_{n})$ by writing $u_i(X^n) = u_{J_n}(X^n) + \pi(i\leq k
< J_n | X^n)$ for all $i\leq J_n$:
\begin{eqnarray*}
nR_{3}(1,J_{n}) & \lesssim & \E_{0}^{(n)}\left(\sum_{j=1}^{J_{n}}\pi(j|X^{n})\sum_{i=1}^{j}a_{i}\tfrac{\tau_{i}^{2}}{\tau_{i}^{2}+\nicefrac{1}{n}}\xi_{i}\right)^{2}\\
 &  & +\E_{0}^{(n)}\left(u_{J_n}(X^n)\sum_{i=1}^{J_{n}}a_{i}\tfrac{\tau_{i}^{2}}{\tau_{i}^{2}+\nicefrac{1}{n}}\xi_{i}\right)^{2}\\
 & := & R_{3,1}(1,J_{n})+R_{3,2}(1,J_{n}).\end{eqnarray*}
Let
$\Gamma_{jn}(X^{n})=\sum_{i=1}^{j}a_{i}\tfrac{\tau_{i}^{2}}{\tau_{i}^{2}+\nicefrac{1}{n}}\xi_{i}$.
We have $\sum_{j=1}^{J_{n}}\pi(j|X^{n})\leq1$ so we can apply
Jensen's inequality, \begin{eqnarray*}
R_{3,1}(1,J_{n}) & \leq & \E_{0}^{(n)}\left(\sum_{j=1}^{J_{n}}\pi(j|X^{n})\Gamma_{jn}(X^{n})^{2}\right)\\
 & \leq & \E_{0}^{(n)}\max_{j\leq J_{n}}\left\{ \Gamma_{jn}(X^{n})^{2}\right\} .\end{eqnarray*}
Noting that $\left(\Gamma_{jn}(X^{n})\right)_{1\leq j\leq J_{n}}$ is
a martingale, we get using Doob's inequality\[
R_{3,1}(1,J_{n})\leq\E_{0}^{(n)}\Gamma_{J_{n}n}(X^{n})^{2}=\sum_{i=1}^{J_{n}}\left(a_{i}\tfrac{\tau_{i}^{2}}{\tau_{i}^{2}+\nicefrac{1}{n}}\right)^{2}\lesssim
J_{n}.\] The second term $R_{3,2}(1,J_{n})$ can be upper bounded in
the same way as for $R_{3}(J_{n},\infty)$ in Equation (\ref{eq:I3}) below by
noting that \[
R_{3,2}(1,J_{n})\lesssim\E_{0}^{(n)}\left[u_{J_{n}}(X^{n})^{2}\left(\sum_{i=K_{n}}^{\infty}\tfrac{\tau_{i}^{2}}{\tau_{i}^{2}+\nicefrac{1}{n}}\left|\xi_{i}\right|\right)^{2}\right].\]
For the upper part $R_{3}(J_{n},\infty)$, we use the bound
(\ref{eq:E_K_n}) on $\E_{0}^{(n)}\left(u_{K_{n}}(X^{n})\right)$,
\begin{eqnarray}
nR_{3}(J_{n},\infty) & \lesssim & \E_{0}^{(n)}\left(\sum_{i=K_{n}}^{\infty}\tfrac{\tau_{i}^{2}}{\tau_{i}^{2}+\nicefrac{1}{n}}u_{i}(X^{n})\left|\xi_{i}\right|\right)^{2}\nonumber \\
 & \lesssim & \E_{0}^{(n)}\left[u_{K_{n}}(X^{n})^{2}\left(\sum_{i=K_{n}}^{\infty}\tfrac{\tau_{i}^{2}}{\tau_{i}^{2}+\nicefrac{1}{n}}\left|\xi_{i}\right|\right)^{2}\right]\label{eq:I3}\\
 & \lesssim & \left[\E_{0}^{(n)}u_{K_{n}}(X^{n})^{4}\right]^{1/2}\left[\E_{0}^{(n)}\left(\sum_{i=K_{n}}^{\infty}\tfrac{\tau_{i}^{2}}{\tau_{i}^{2}+\nicefrac{1}{n}}\left|\xi_{i}\right|\right)^{4}\right]^{1/2}\nonumber \\
 & \lesssim & \left[\E_{0}^{(n)}u_{K_{n}}(X^{n})\right]^{1/2}\left[\left(\sum_{i=K_{n}}^{\infty}\tfrac{\tau_{i}^{2}}{\tau_{i}^{2}+\nicefrac{1}{n}}\right)^{4}\right]^{1/2}\nonumber \\
 & \lesssim & e^{-c_{2}K_{n}/2}n^{1/q},\nonumber \end{eqnarray}
where we bound the different moments of $\left|\xi_{i}\right|$ by a
unique constant and then use
$\sum_{i=K_{n}}^{\infty}\tau_{i}^{2}/(\tau_{i}^{2}+\nicefrac{1}{n})=\go(n^{1/2q})$.
Then $R_{3}=\go(n^{-(2\beta-1)/2\beta})$.

To sum up, $R_{2}(1,J_{n})$ is the only remaining term. We build an
example where it is of greater order than
$n^{-(2\beta-1)/2\beta}$. Let $\boldsymbol{\theta}_{0}$ be
defined  by its coordinates $\theta_{0i}
=i^{-\beta-1/2}\left(\log(i+1)\right)^{-1}$ such that the series
$\sum_{i}\theta_{0i}^{2}i^{2\beta}$ converge, so
$\boldsymbol{\theta}_{0}$ belongs to the Sobolev ball of smoothness
$\beta$. It is assumed that $a_{i}=\psi_{i}(t)=1$, so all terms in the
sum $R_{2}(1,J_{n})$ are \textit{positive}, hence\[
R_{2}(1,J_{n})\geq \frac{1}{4}
\E_{0}^{(n)}\left(\sum_{i=K_{n}}^{J_{n}}(1-u_{i}(X^{n}))\theta_{0i}\right)^{2},\]
noting that for $i\leq J_{n}$, we have $n\tau_{i}^{2}\geq
n^{1-q/\beta}\geq1$ because $q\leq\beta$ and $n\geq1$, so
$\tau_{i}^{2}/(\tau_{i}^{2}+\nicefrac{1}{n})\geq1/2$.
Moreover, $u_{i}(X^{n})$  decreases with $i$, so
\begin{equation*}
R_{2}(1,J_{n}) \geq \frac{1}{4}
\E_{0}^{(n)}\left((1-u_{K_n}(X^{n}))^2\right)\left(\sum_{i=K_{n}}^{J_{n}}\theta_{0i}\right)^{2},\end{equation*}
where $\E_{0}^{(n)}\left((1-u_{K_n}(X^{n}))^2\right)$ is lower
bounded by a positive constant for $n$ large enough. Comparing the
 series $\sum_{i=K_{n}}^{J_{n}}\theta_{0i}$ to an integral
shows that it is bounded from below by $K_{n}^{-\beta+1/2}/\log n$.
We obtain by using Lemma \ref{lem:lowerbound} that
$R_{n}^{\text{loc}}(\boldsymbol{\theta}_{0},t)=R_{2}(1,J_{n})(1+o(1))
\gtrsim{n^{-\frac{2\beta-1}{2\beta+1}}}/{\log^{2}n}$, which ends the
proof.
\end{proof}
\end{document}